\documentclass[final]{siamart1116}

\usepackage{mathtools}
\usepackage{amssymb}
\usepackage{stmaryrd}
\usepackage{color}
\usepackage{pdflscape}
\usepackage{epsfig}

\usepackage{epstopdf}
\usepackage{multirow}
\usepackage{hhline}
\usepackage{txfonts}
\usepackage{mathrsfs}
\usepackage{color}
\usepackage{algorithmic}

\usepackage{subcaption}
\usepackage{mathrsfs}

\usepackage{tikz}
\usepackage{graphicx}
\usetikzlibrary{positioning}

\newtheorem{thm}{Theorem}[section]

\newtheorem{lem}[thm]{Lemma}
\newtheorem{rem}[thm]{Remark}
\newtheorem{define}[thm]{Definition}

\numberwithin{equation}{section}
\title{Fractional Sensitivity Equation Method: \\ Applications to Fractional Model Construction \thanks{This work was supported by the AFOSR Young Investigator Program (YIP) award (FA9550-17-1-0150).}}

\author{
	Ehsan Kharazmi
	\footnote{D\lowercase{epartment of} C\lowercase{omputational} M\lowercase{athematics}, S\lowercase{cience}, \lowercase{and}, E\lowercase{ngineering} \& D\lowercase{epartment of} M\lowercase{echanical} E\lowercase{ngineering},	
		M\lowercase{ichigan} S\lowercase{tate} U\lowercase{niversity}, 428 S S\lowercase{haw} L\lowercase{ane}, E\lowercase{ast} L\lowercase{ansing}, MI 48824, USA}
	, Mohsen Zayernouri
	\footnote{D\lowercase{epartment of} C\lowercase{omputational} M\lowercase{athematics}, S\lowercase{cience}, \lowercase{and}, E\lowercase{ngineering} \&
		D\lowercase{epartment of} M\lowercase{echanical} E\lowercase{ngineering},	
		M\lowercase{ichigan} S\lowercase{tate} U\lowercase{niversity}, 428 S S\lowercase{haw} L\lowercase{ane}, E\lowercase{ast} L\lowercase{ansing}, MI 48824, USA,  C\lowercase{orresponding author; zayern@msu.edu}}
}

\begin{document}
\maketitle
\begin{abstract}
Fractional differential equations provide a tractable mathematical framework to describe anomalous behavior in complex physical systems, yet they introduce new sensitive model parameters, i.e. derivative orders, in addition to model coefficients. We formulate a sensitivity analysis of fractional models by developing a  \textit{fractional sensitivity equation method}. We obtain the adjoint fractional sensitivity equations, in which we present a fractional operator associated with \textit{logarithmic-power law} kernel. We further construct a gradient-based optimization algorithm to compute an accurate parameter estimation in fractional model construction. We develop a fast, stable, and convergent Petrov-Galerkin spectral method to numerically solve the coupled system of original fractional model and its corresponding adjoint fractional sensitivity equations. 

\end{abstract}
%
%\begin{AMS}
%\end{AMS}
%

%================================================================================================================================================
%================================================================================================================================================

\pagestyle{myheadings}
\thispagestyle{plain}

%
%%%%%%%%%%%%%%%%%%%%%%%%%%%
\section{Introduction}
\label{Sec: Introduction}
%%%%%%%%%%%%%%%%%%%%%%%%%%%
%
The experimental observations in a divers number of complex physical systems reveal ubiquitous anomalous behavior in the associated underlying processes \cite{West2003,west2016fractional}, where the anomaly manifests itself in skewness and sharpness of heavy tailed distributions. Fractional differential equations, which generalize their integer order counterparts, construct a rigorous mathematical framework to formulate models that flawlessly describe such anomalies. The excellence of fractional operator in accurate prediction of non-locality and memory effects is the inherent non-local nature of singular power-law kernel, whose order is defined as fractional derivative order, i.e. fractional index. Theses operators are being extensively used in analysis and design of models for a wide range of multi-scale multi-physics phenomena. Examples include viscoelastic materials and wave propagation  \cite{Mainardi2010,suzuki2016fractional,meral2010fractional}, non-Brownian transport phenomena in porous media and disordered materials \cite{meer01, meral2010fractional}, non-Newtonian complex fluids with multi-phase applications and rheology \cite{jaishankar2014,sreenivasan1997phenomenology, jha2003evidence, Castillo2004Plasma,jaishankar2013}, multi-scale patterns 
in biological tissue \cite{naghibolhosseini2015estimation,naghibolhosseini2018fractional,magin2010fractional,anastasio1994fractional,djordjevic2003fractional}, chaos/fractals and automatic control \cite{le1991fractal,duarte2002chaotic}. However, the key challenges of such models are the excessive computational cost in numerically integrating the convolution operation, and more importantly, introducing fractional derivative orders as extra model parameters, whose values are essentially obtained from experimental data. The sensitivity assessment of fractional models with respect to fractional indecis can build a bridge between experiments and mathematical models to gear observable data via proper optimization techniques, and thus, systematically improve the existing models in both analysis and design approaches. We formulate a mathematical framework by developing a \textit{fractional sensitivity equation method}, where we investigate the response sensitivity of fractional differential equations with respect to model parameters including derivative orders, and further construct an iterative algorithm in order to exploit the obtained sensitivity field in parameter estimation.

\vspace{0.1 in}
%
%%%%%%%%%%%%%%%%%%%%%%%%%%%
\noindent\textbf{Fractional Sensitivity Analysis}.
%%%%%%%%%%%%%%%%%%%%%%%%%%%
%
Sensitivity assessment approaches are commonly categorized as, finite difference, continuum and discrete derivatives, and computational or automatic differentiation, where the sensitivity coefficients are generally defined as partial derivative of corresponding functions (model output) with respect to design/analysis parameters of interest. Finite difference schemes use a first order Taylor series expansion to approximate the sensitivity coefficients, where accuracy depends strongly on step increment \cite{Martins00,Sobieski90}. Continuum and discrete derivative techniques however, differentiate the system response with respect to parameters, where the former, which is also known as sensitivity equation method (SEM, see \cite{liu2016two,zayernouri2011coherent} and references therein), directly computes the derivatives and obtain a set of (coupled) adjoint continuum sensitivity equations; while the latter performs differentiation after discretization of original equation \cite{stanford2010adjoint}. Automatic differentiation method also refers to a differentiation of the computer code \cite{Bischof96,Bischof97,Bischof03}. Fig.2 in \cite{van2005review} provides a descriptive schematic of these different approaches. We extend the continuum derivative technique to develop a \textit{fractional sensitivity equation method} (FSEM) in the context of fractional partial differential equations (FPDEs). To formulate the sensitivity analysis framework, we let $q$ be a set of model parameters including fractional indices and obtain the adjoint \textit{fractional sensitivity equations} (FSEs) by taking the partial derivative of FPDE with respect to $q$. These adjoint equations introduce a new fractional operator, associated with the \textit{logarithmic-power law} kernel, which to best of our knowledge has been presented for the first time here in the context of fractional sensitivity analysis. The key property of derived FSEs is that they preserve the structure of original FPDE. Thus, similar discretization scheme and forward solver can be readily applied with a minimal required changes.

\vspace{0.1 in}
%
%%%%%%%%%%%%%%%%%%%%%%%%%%%
\noindent\textbf{Model Construction: Estimation of Fractional Indices}.
%%%%%%%%%%%%%%%%%%%%%%%%%%%
%
Several numerical methods have been developed to solve inverse problem of model construction from available experimental observations or synthetic data. They typically convert the problem of model parameter estimation into an optimization problem, and then, formulate a suitable estimator by minimizing an objective function. These methods are stretched over but no limited to perturbation methods \cite{wei2010coupled}, weighted least squares approach \cite{chakraborty2009parameter,cho2015fractional,kelly2017fracfit}, nonlinear regression \cite{lim2014parameter}, and Levenberg-Marquardt method \cite{ghazizadeh2012inverse,chen2016fast,yu2016numerical,yu2017numerical}. We develop a bi-level FSEM-based parameter estimation method in order to construct fractional models, in a sense that the method obtains model coefficients in one level, and then searches for estimate of fractional indices in the next level. We formulate the optimization problem by defining objective functions as two types of model error that measures the difference in computed output/input of fractional model with true output/input in an $L^2$-norm sense. We further formulate a gradient-based minimizer, employing developed FSEM, and propose a two-stage search algorithm, namely, coarse grid searching and nearby solution. The first stage construct a crude manifold of model error over a coarse discretization of parameter space to locate a local neighborhood of minimum, and the second stage uses the gradient decent method in order to converge to the minimum point.

\vspace{0.1 in}
%
%%%%%%%%%%%%%%%%%%%%%%%%%%%
\noindent\textbf{Discretization Scheme}.
%%%%%%%%%%%%%%%%%%%%%%%%%%%
%
The iterative nature of parameter estimators instruct simulation of fractional model at each iteration step of model parameters. Therefore, one of the major tasks in computational model construction is to develop numerical methods that can efficiently discretize the physical domain and accurately solve the fractional model. The sensitivity framework additionally raise the complication by rendering coupled systems of FPDE and adjoint FSEs, and thus, demanding more versatile schemes. In addition to numerous finite difference methods \cite{Gorenflo2002, Sun2006,  Lin2007, wang2010direct, wang2011fast, Cao2013, zeng2015numerical, zayernouri2016_JCP_Frac_AB_AM}, recent works have elaborated efficient spectral schemes, for discretizing FPDEs in physical domain, see e.g., \cite{Rawashdeh2006, Lin2007, Khader2011, Khader2012, Li2009, Li2010, chen2014generalized, wang2015high, bhrawy2015spectral}. More recently, Zayernouri et al. \cite{Zayernouri2013, zayernouri2015tempered} developed two new spectral theories on fractional and tempered fractional Sturm-Liouville problems, and introduced explicit corresponding eigenfunctions, namely Jacobi poly-fractonomials of first and second kind. These eignefunctions are comprised of smooth and fractional parts, where the latter can be tunned to capture singularities of true solution. They are successfully employed in constructing discrete solution/test function spaces and developing a series of high-order and efficient Petrov-Galerkin spectral methods, see \cite{Zayernouri_FODEs_2014, Zayernouri14-SIAM-Frac-Delay, Zayernouri14-SIAM-Frac-Advection,suzuki2016fractional,Samiee2016unified,samiee2017fast,kharazmi2017petrov,kharazmi2017sem,kharazmi2018fractional,lischke2017petrov,samiee2017unified,samiee2018petrov}. We formulate a numerical scheme in solving coupled system of FPDE and adjoint FSEs by extending the mathematical framework in \cite{Samiee2016unified} and accommodating extra required regularity in the underlying function spaces. We employ Jacobi poly-fractonomials and Legendre polynomials as temporal and spatial basis/test functions, respectively, to develop a Petrov-Galerkin (PG) spectral method. The smart choice of coefficients in spatial basis/test functions yields symmetric property in the resulting mass/stiffness matrices, which is then exploited to formulate a fast solver. Following similar procedure as in \cite{Samiee2016unified}, we also show that the coupled system is mathematically well-posed, and the proposed numerical scheme is stable.

\vspace{0.1 in}

The rest of paper is organized as follows. In section \ref{Sec: Fractional Calculus}, we recall some preliminary definitions in fractional calculus and define proper solution/test spaces beside useful lemmas. In section \ref{Sec: Problem Definition}, we define the problem by providing the fractional model, and then take the weak form of the problem as well. We develop FSEM in section \ref{Sec: FSEM} for the case of FIVP and FPDE. We define the underlying mathematical frame work for the coupled system of FPDE and FSEs and also  construct our Petrov-Galerkin spectral numerical scheme. Moreover, we develop the FSEM based model construction algorithm in section \ref{Sec: FSEM Model Construction} and finally, provide the numerical results in section \ref{Sec: numerical results part I}. We conclude the paper with a summary and conclusion.

%%
%\newpage

%
%%%%%%%%%%%%%%%%%%%%%%%%%%%
\section{Definitions}
\label{Sec: Fractional Calculus}
%%%%%%%%%%%%%%%%%%%%%%%%%%%
%
Let $ \xi \in [-1,1]$. The left- and right-sided fractional derivative of order $\sigma$, $n-1 < \sigma \leq n$, $n \in \mathbb{N}$, are defined as (see e.g., \cite{Miller93, Podlubny99})
\begin{align}
\label{Eq: left RL derivative}
(\prescript{RL}{-1}{\mathcal{D}}_{\xi}^{\sigma}) u(\xi) = \frac{1}{\Gamma(n-\sigma)}  \frac{d^{n}}{d\xi^n} \int_{-1}^{\xi} \frac{u(s) ds}{(\xi - s)^{\sigma+1-n} },\quad \xi >-1 ,
\\
\label{Eq: right RL derivative}
(\prescript{RL}{\xi}{\mathcal{D}}_{1}^{\sigma}) u(\xi) = \frac{1}{\Gamma(n-\sigma)}  \frac{(-d)^{n}}{d\xi^n} \int_{\xi}^{1} \frac{u(s) ds}{(s - \xi)^{\sigma+1-n} },\quad \xi < 1 ,
\end{align}
respectively. 
An alternative approach in defining the fractional derivatives is the left- and right-sided Caputo derivatives of order $\sigma$, defined, as
\begin{align}
\label{Eq: left Caputo derivative}
&
(\prescript{C}{-1}{\mathcal{D}}_{\xi}^{\sigma} u) (\xi) = \frac{1}{\Gamma(n-\sigma)}  \int_{-1}^{\xi} \frac{u^{(n)}(s) ds}{(\xi - s)^{\sigma+1-n} },\quad \xi>-1,
\\
\label{Eq: right Caputo derivative}
&
(\prescript{C}{\xi}{\mathcal{D}}_{1}^{\sigma} u) (\xi) =  \frac{1}{\Gamma(n-\sigma)}  \int_{\xi}^{1} \frac{u^{(n)}(s) ds}{(s-\xi)^{\sigma+1-n} },\quad \xi<1.
\end{align}
By performing an affine mapping from the standard domain $[-1,1]$ to the interval $t \in [a,b]$, we obtain
\begin{eqnarray}
\label{Eq: RL in xL-xR}
\prescript{RL}{a}{\mathcal{D}}_{t}^{\sigma} u  &=&  (\frac{2}{b-a})^\sigma (\prescript{RL}{-1}{\mathcal{D}}_{\xi}^{\sigma} \, u )(\xi), 
\\ 
\label{Eq: Caputo in xL-xR}
\prescript{C}{a}{\mathcal{D}}_{t}^{\sigma} u  &=&  (\frac{2}{b-a})^\sigma (\prescript{C}{-1}{\mathcal{D}}_{\xi}^{\sigma} \, u) (\xi).
\end{eqnarray} 
Hence, we can perform the operations in the standard domain only once for any given $\sigma$ and efficiently utilize them on any arbitrary interval without resorting to repeating the calculations. Moreover, the corresponding relationship between the Riemann-Liouville and Caputo fractional derivatives in $[a,b]$ for any $\sigma \in (0,1)$ is given by 
\begin{equation}
\label{Eq:  Caputo vs. Riemann}
(\prescript{RL}{a}{\mathcal{D}}_{t}^{\sigma} \, u) (t)  =  \frac{ u(a)}{\Gamma(1-\sigma) (t-a)^{\sigma}}  +   (\prescript{C}{a}{\mathcal{D}}_{t}^{\sigma} \, u) (t).
\end{equation}

\begin{define} 
\textup{We define the following left- and right-sided integro-differential operator with logarithmic-power law kernel, namely \textit{Log-Pow integro-differential operator}, given as,}
\end{define}
\begin{align}
\label{Eq: left sided LP kernel derivative}
&
\prescript{RL-LP}{a}{\mathcal{D}}_{x}^{\sigma} u (x) = 
\frac{1}{\Gamma(n-\sigma)}  \frac{d^n}{dx^n} \int_{a}^{x} \frac{ \log(x-s) \, u(s)}{(x - s)^{\sigma -n +1}} ds,
\\
\label{Eq: right sided LP kernel derivative}
&
\prescript{RL-LP}{x}{\mathcal{D}}_{b}^{\sigma} u (x) = 
\frac{1}{\Gamma(n-\sigma)}  \frac{(-d)^n}{dx^n} \int_{x}^{b} \frac{ \log(s-x) \, u(s)}{(s - x)^{\sigma-n+1}} ds,
\\
\label{Eq: left sided LP kernel derivative Caputo}
&
\prescript{C-LP}{a}{\mathcal{D}}_{x}^{\sigma} u (x) = 
\frac{1}{\Gamma(n-\sigma)}  \, \int_{a}^{x} \frac{ \log(x-s) \, u^{(n)}(s)}{(x - s)^{\sigma -n +1}} ds,
\\
\label{Eq: right sided LP kernel derivative Caputo}
&
\prescript{C-LP}{x}{\mathcal{D}}_{b}^{\sigma} u (x) = 
\frac{1}{\Gamma(n-\sigma)}  \, \int_{x}^{b} \frac{ \log(s-x) \, u^{(n)}(s)}{(s - x)^{\sigma-n+1}} ds,
\end{align}
where $RL-LP$ and $C-LP$ stand for Log-Pow integro-differential operator, which partially resemble the fractional derivative in Riemann-Liouville and Caputo sense, respectively. The following lemma shows a useful relation between the two aforementioned operators.

%\vspace{0.15 cm}
%
\begin{lem}
	\label{lem: Caputo vs. Riemann PowLog}
	Let $x\in [a,b]$. Then, the following relation holds.
	
	\noindent Part A: $\sigma \in (0,1)$ 
	\begin{align}
	\label{Eq:  Caputo vs. Riemann PowLog 1}
	\prescript{RL-LP}{a}{\mathcal{D}}_{x}^{\sigma} \, u (x)  =  \frac{ u(a)}{\Gamma(1-\sigma)}\frac{\log(x-a)}{(x-a)^{\sigma}}  +   \prescript{C-LP}{a}{\mathcal{D}}_{x}^{\sigma} \, u (x).
	\end{align}

	\noindent Part B: $\sigma \in (1,2)$
	\begin{align}
	\label{Eq:  Caputo vs. Riemann PowLog 2}
	\prescript{RL-LP}{a}{\mathcal{D}}_{x}^{\sigma} \, u (x)  =  \frac{u(a)}{\Gamma(2-\sigma)} \frac{1+(1-\sigma)\log(x-a)}{(x-a)^{\sigma}} + \frac{u'(a)}{\Gamma(2-\sigma)} \frac{\log(x-a)}{(x-a)^{\sigma-1}}  +   \prescript{C-LP}{a}{\mathcal{D}}_{x}^{\sigma} \, u (x).
	\end{align}
\end{lem}
\begin{proof}
	See Appendix \ref{Sec: App. proof of RL Caputo PL derivative} for proof.
\end{proof}
%

%
%%%%%%%%%%%%%%%%%%%%%%%%%%%%%
\subsection{\textbf{Fractional Sobolev Spaces}}
\label{Sec: Sol Test Space FPDE}
%%%%%%%%%%%%%%%%%%%%%%%%%%%%%
%
We define some functional spaces and their associated norms \cite{kharazmi2017petrov,Li2009}. By $H^\sigma(\mathbb{R}) = \big{\{} u(t) \vert u \in L^{2}(\mathbb{R});\, (1+\vert \omega \vert^2)^{\frac{\sigma}{2}} F(u)(\omega) \in L^{2}(\mathbb{R}) \big{\}}$, $\sigma \geq 0$, we denote the fractional Sobolev space on $\mathbb{R}$, endowed with norm $\Vert u \Vert_{H^{\sigma}_{\mathbb{R}}}=\Vert (1+\vert \omega \vert^2)^{\frac{\sigma}{2}} F(u)(\omega) \Vert_{L^{2}(\mathbb{R})}$, where $\mathcal{F}(u)$ represents the Fourier transform of $u$. Subsequently, we denote by $H^\sigma(\Lambda) =\big{\{} u\in L^{2}(\Lambda)\, \vert \,\exists \tilde{u} \in H^{\sigma}(\mathbb{R})\, \, s.t. \, \,\tilde{u}\vert_{\Lambda}=u \big{\}}$, $\sigma \geq 0$, the fractional Sobolev space on any finite closed interval, e.g. $\Lambda = (a,b)$, with norm $ \Vert u \Vert_{H^{\sigma}(\Lambda)}= \underset{\tilde{u}\in H^{\sigma}_{\mathbb{R}},\, \tilde{u}\vert_{\Lambda}=u }{\inf} \, \Vert \tilde{u} \Vert_{{}H^{\sigma}(\mathbb{R})}$. We define the following useful norms as:
\begin{align*}
&\Vert \cdot \Vert_{{^l}H^{\sigma}_{}(\Lambda)} = \Big(\Vert \prescript{}{a}{\mathcal{D}}_{x}^{\sigma}\, (\cdot)\Vert_{L^2(\Lambda)}^2+\Vert \cdot \Vert_{L^2(\Lambda)}^2 \Big)^{\frac{1}{2}},
\\
&\Vert \cdot \Vert_{{^r}H^{\sigma}_{}(\Lambda)} = \Big(\Vert \prescript{}{x}{\mathcal{D}}_{b}^{\sigma}\, (\cdot)\Vert_{L^2(\Lambda)}^2+\Vert \cdot \Vert_{L^2(\Lambda)}^2 \Big)^{\frac{1}{2}},
\\
&\Vert \cdot \Vert_{{^c}H^{\sigma}_{}(\Lambda)} = \Big(\Vert \prescript{}{x}{\mathcal{D}}_{b}^{\sigma}\, (\cdot)\Vert_{L^2(\Lambda)}^2+\Vert \prescript{}{a}{\mathcal{D}}_{x}^{\sigma}\, (\cdot)\Vert_{L^2(\Lambda)}^2+\Vert \cdot \Vert_{L^2(\Lambda)}^2 \Big)^{\frac{1}{2}},
\end{align*}
where the equivalence of $\Vert \cdot \Vert_{{^l}H^{\sigma}_{}(\Lambda)}$ and $\Vert \cdot \Vert_{{^r}H^{\sigma}_{}(\Lambda)}$ are shown in \cite{Li2009,ervin2007variational,Li2010}. We show the equivalence of these two norms with $\Vert \cdot \Vert_{{^c}H^{\sigma}_{}(\Lambda)}$ in the following lemma.
%
%\vspace{0.1 in}
%
\begin{lem}
	\label{Lem: norm equivalence 1}
	%	\label{lemma31}
	Let $\sigma \geq 0$ and $\sigma \neq n-\frac{1}{2}$. Then, the norms $\Vert \cdot \Vert_{{^l}H^{\sigma}_{}(\Lambda)}$ and $\Vert \cdot \Vert_{{^r}H^{\sigma}_{}(\Lambda)}$ are equivalent to $\Vert \cdot \Vert_{{^c}H^{\sigma}_{}(\Lambda)}$.
\end{lem}
\begin{proof}
	See Appendix \ref{Sec: App. proof of norm equivalence} for proof.
\end{proof}

We also define $C^{\infty}_{0}(\Lambda)$ as the space of smooth functions with compact support in $(a,b)$. We denote by $\prescript{l}{}H^{\sigma}_0(\Lambda)$, $\prescript{r}{}H^{\sigma}_0(\Lambda)$, and $\prescript{c}{}H^{\sigma}_0(\Lambda)$ as the closure of $C^{\infty}_{0}(\Lambda)$ with respect to the norms $\Vert \cdot \Vert_{{^l}H^{\sigma}_{}(\Lambda)}$, $\Vert \cdot \Vert_{{^r}H^{\sigma}_{}(\Lambda)}$, and $\Vert \cdot \Vert_{{^c}H^{\sigma}_{}(\Lambda)}$. 
\begin{lem}[\cite{ervin2007variational,Li2010}]
	\label{Lem: sobolev equivalence}
	The Sobolev spaces $\prescript{l}{}H^{\sigma}_0(\Lambda)$, $\prescript{r}{}H^{\sigma}_0(\Lambda)$, and $\prescript{c}{}H^{\sigma}_0(\Lambda)$ are equal and their seminorms are equivalent to $\vert \cdot \vert_{{}H^{\sigma}_{}(\Lambda)}^{*} = \big\vert \left( \prescript{}{a}{\mathcal{D}}_{x}^{\sigma}\, (\cdot), \prescript{}{x}{\mathcal{D}}_{b}^{\sigma}\, (\cdot) \right) \big\vert_{\Lambda}^{\frac{1}{2}}$
\end{lem}

Based on Lemma \eqref{Lem: sobolev equivalence}, and assuming that $\big{\vert}(\prescript{}{a}{\mathcal{D}}_{x}^{\sigma}\, u, \prescript{}{x}{\mathcal{D}}_{b}^{\sigma}\, v )_{\Lambda}^{} \big{\vert} > 0$ and $\big{\vert}(\prescript{}{x}{\mathcal{D}}_{b}^{\sigma}\, u, \prescript{}{a}{\mathcal{D}}_{x}^{\sigma}\, v )_{\Lambda}^{} \big{\vert} > 0$, we can prove that $\big{\vert}(\prescript{}{a}{\mathcal{D}}_{x}^{\sigma}\, u, \prescript{}{x}{\mathcal{D}}_{b}^{\sigma}\, v )_{\Lambda}^{} \big{\vert} \geq \beta_1 \, \vert u \vert_{{^l}H^{\sigma}_{}(\Lambda)}\, \vert  v \vert_{{^r}H^{\sigma}_{}(\Lambda)}$ and $\big{\vert}(\prescript{}{x}{\mathcal{D}}_{b}^{\sigma}\, u, \prescript{}{a}{\mathcal{D}}_{x}^{\sigma}\, v )_{\Lambda}^{} \big{\vert} \geq \beta_2 \, \vert u \vert_{{^r}H^{\sigma}_{}(\Lambda)}\, \vert  v \vert_{{^l}H^{\sigma}_{}(\Lambda)}$, where $\beta_1$ and $\beta_2$ are positive constants. Following \cite{Samiee2016unified}, we define the corresponding solution and test spaces of our problem. Thus, by letting $\Lambda_1 = (a_1,b_1)$, $\Lambda_j = (a_j,b_j) \times \Lambda_{j-1}$ for $j=2,\cdots,d$, we define $\mathcal{X}_1 = H^{\frac{\beta_1}{2}}_{0}(\Lambda_1)$, which is associated with the norm $ \Vert \cdot \Vert_{{^c}H^{\frac{\beta_1}{2}}_{}(\Lambda_1)}$, and accordingly, $\mathcal{X}_j, \, j=2,\cdots,d$ as
\begin{eqnarray}
\mathcal{X}_2 &=& H^{\frac{\beta_2}{2}}_0 \Big((a_2,b_2); L^2(\Lambda_1) \Big) \cap L^2((a_2,b_2); \mathcal{X}_1),
\\
&\vdots&
\nonumber
\\
\mathcal{X}_d &=& H^{\frac{\beta_d}{2}}_0 \Big((a_d,b_d); L^2(\Lambda_{d-1}) \Big) \cap L^2((a_d,b_d); \mathcal{X}_{d-1}),
\end{eqnarray}
associated with norms $\Vert \cdot \Vert_{\mathcal{X}_j} = \bigg{\{} \Vert \cdot \Vert_{H^{\frac{\beta_j}{2}}_0 \Big((a_j,b_j); L^2(\Lambda_{j-1}) \Big)}^2 + \Vert \cdot \Vert_{ L^2\Big((a_j,b_j); \mathcal{X}_{j-1}\Big)}^2 \bigg{\}}^{\frac{1}{2}}, \,\, j=2,3,\cdots,d$.

\vspace{0.1 in}
\begin{lem}
	%
	%	\label{norm_221}
	\label{space norm 1}
	Let $\sigma \geq 0$ and $\sigma \neq n-\frac{1}{2}$. Then, for $j=1,2,\cdots,d$ 
	\begin{align*}
	%
	%	\label{norm_Xd_2}
	\Vert \cdot \Vert^{2}_{\mathcal{X}_j} \equiv  \sum_{i=1}^{j} \Big(\Vert \prescript{}{x_i}{\mathcal{D}}_{b_i}^{\beta_i/2}\, (\cdot)\Vert_{L^2(\Lambda_j)}^2+\Vert \prescript{}{a_i}{\mathcal{D}}_{x_i}^{\beta_i/2}\, (\cdot)\Vert_{L^2(\Lambda_j)}^2 \Big) + \Vert  \cdot \Vert_{L^2(\Lambda_j)}^2.
	\end{align*}
\end{lem}
\begin{proof}
	See Appendix \ref{Sec: App. proof of norm Xd} for proof.
\end{proof}
\vspace{0.1 in}

Moreover, by letting ${_0}C^{\infty}(I)$ and $C^{\infty}_{0}(I)$ be the space of smooth functions with compact support in $(0,T]$ and $[0,T)$, respectively, we define $\prescript{l}{}H^{s}(I)$ and $\prescript{r}{}H^{s}(I)$ as the closure of ${_0}C^{\infty}(I)$ and $C^{\infty}_{0}(I)$ with respect to the norms $\Vert \cdot \Vert_{\prescript{l}{}H^{s}(I)}$ and $\Vert \cdot \Vert_{\prescript{r}{}H^{s}(I)}$. 
We also define 
\begin{align*}
\prescript{l}{0}H^{\frac{\alpha}{2}} \Big(I; L^2(\Lambda_d) \Big) = 
\Big{\{} u \, \big|\, \Vert u(t,\cdot) \Vert_{L^2(\Lambda_d)} \in H^{\frac{\alpha}{2}}(I), u\vert_{t=0}=u\vert_{x=a_j}=u\vert_{x=b_j}=0,\, j=1,2,\cdots,d  \Big{\}},
\\
\prescript{r}{0}H^{\frac{\alpha}{2}} \Big(I; L^2(\Lambda_d) \Big) = 
\Big{\{} v \,\big|\, \Vert v(t,\cdot) \Vert_{L^2(\Lambda_d)} \in H^{\frac{\alpha}{2}}(I), v\vert_{t=T}=v\vert_{x=a_j}=v\vert_{x=b_j}=0,\, j =1,2,\cdots,d  \Big{\}},
\end{align*}
equipped with norms $\Vert u \Vert_{\prescript{l}{}H^{\frac{\alpha}{2}}(I; L^2(\Lambda_d))}$ and $\Vert u \Vert_{\prescript{r}{}H^{\frac{\alpha}{2}}(I; L^2(\Lambda_d))}$, respectively, which take the following forms
\begin{align}
%
%\label{norm_222}
\label{space norm 2}
&
\Vert u \Vert_{\prescript{l}{}H^{\frac{\alpha}{2}}(I; L^2(\Lambda_d))} = \Big{\Vert} \, \Vert u(t,\cdot) \Vert_{L^2(\Lambda_d)}\, \Big{\Vert}_{{^l}H^{\frac{\alpha}{2}}(I)}
=\Big(\Vert \prescript{}{0}{\mathcal{D}}_{t}^{\frac{\alpha}{2}}\, (u)\Vert_{L^2(\Omega)}^2 + \Vert u\Vert_{L^2(\Omega)}^2 \Big)^{\frac{1}{2}},
\\
&\Vert u \Vert_{\prescript{r}{}H^{\frac{\alpha}{2}}(I; L^2(\Lambda_d))} = \Big{\Vert} \, \Vert u(t,\cdot) \Vert_{L^2(\Lambda_d)}\, \Big{\Vert}_{{^r}H^{\frac{\alpha}{2}}(I)}
= \Big(\Vert \prescript{}{t}{\mathcal{D}}_{T}^{\frac{\alpha}{2}}\, (u)\Vert_{L^2(\Omega)}^2+\Vert u\Vert_{L^2(\Omega)}^2\Big)^{\frac{1}{2}}.
\end{align}
%

%
%%%%%%%%%%%%%%%%%%%%%%%%%%%%%
\subsection*{\textbf{Solution and Test Spaces}}
%%%%%%%%%%%%%%%%%%%%%%%%%%%%%
%
We define the solution space $U$ and test space $V$, respectively, as
\begin{align}
\label{Eq: solution test space}
U  = \prescript{l}{0}H^{\frac{\alpha}{2}}\Big(I; L^2(\Lambda_d) \Big) \cap L^2(I; \mathcal{X}_d),
\quad%
V  = \prescript{r}{0}H^{\frac{\alpha}{2}}\Big(I; L^2(\Lambda_d)\Big) \cap L^2(I; \mathcal{X}_d),
\end{align}
endowed with norms
\begin{align}
\label{Eq: solution test space norm}
\Vert u \Vert_{U} &= 
\Big{\{}\Vert u \Vert_{\prescript{l}{}H^{\frac{\alpha}{2}}(I; L^2(\Lambda_d))}^2 + \Vert u \Vert_{L^2(I; \mathcal{X}_d)}^2 \Big{\}}^{\frac{1}{2}}, 
\quad
\Vert v \Vert_{V} &= 
\Big{\{}\Vert v \Vert_{\prescript{r}{}H^{\tau}(I; L^2(\Lambda_d))}^2  + \Vert v \Vert_{ L^2(I; \mathcal{X}_d)}^2 \Big{\}}^{\frac{1}{2}},
\end{align}
Using Lemma \ref{space norm 1}, we can show that
\begin{align}
%
%\label{norm_2221}
\label{space norm 3}
\Vert u \Vert_{L^2(I; \mathcal{X}_d)}
=
\Big{\Vert} \, \Vert u(t,.) \Vert_{\mathcal{X}_d}\,\Big{\Vert}_{L^2(I)}
=
\Big{\{}  \Vert u \Vert_{L^2(\Omega)}^2 + \sum_{j=1}^{d} \big( \Vert \prescript{}{x_j}{\mathcal{D}}_{b_j}^{\frac{\beta_j}{2}}\, u  \,\Vert_{L^2(\Omega)}^2 
+ \Vert \prescript{}{a_j}{\mathcal{D}}_{x_j}^{\frac{\beta_j}{2}}\, u \,\Vert_{L^2(\Omega)}^2 \big) 
\Big{\}}^{\frac{1}{2}}.
\end{align}
Therefore, by \eqref{space norm 2} we write \eqref{Eq: solution test space norm} as
\begin{align}
\label{Eq: solution space norm}
\Vert u \Vert_{U} 
&= 
\Big{\{}  \Vert u \Vert_{L^2(\Omega)}^2 + \Vert \prescript{}{0}{\mathcal{D}}_{t}^{\frac{\alpha}{2}}\, u\, \Vert_{L^2(\Omega)}^2 
+ \sum_{j=1}^{d} \big( \Vert \prescript{}{x_j}{\mathcal{D}}_{b_j}^{\frac{\beta_j}{2}}\, u \, \Vert_{L^2(\Omega)}^2+\Vert \prescript{}{a_j}{\mathcal{D}}_{x_j}^{\frac{\beta_j}{2}}\, u \, \Vert_{L^2(\Omega)}^2 \big) \Big{\}}^{\frac{1}{2}}, 
\\
\label{Eq: test space norm}
\Vert v \Vert_{V} 
& =
\Big{\{}  \Vert v \Vert_{L^2(\Omega)}^2 + \Vert \prescript{}{t}{\mathcal{D}}_{T}^{\frac{\alpha}{2}}\, v \, \Vert_{L^2(\Omega)}^2 
+ \sum_{j=1}^{d} \big( \Vert \prescript{}{x_j}{\mathcal{D}}_{b_j}^{\frac{\beta_j}{2}}\, v \, \Vert_{L^2(\Omega)}^2
+\Vert \prescript{}{a_j}{\mathcal{D}}_{x_j}^{\frac{\beta_j}{2}}\, v \, \Vert_{L^2(\Omega)}^2 \big) \Big{\}}^{\frac{1}{2}}.
\end{align}

The following lemmas help us obtain the weak formulation of our problem, construct the numerical scheme and further prove the stability of our method. 

%\vspace{0.2 cm}
%
\begin{lem}
	\label{Lem: left frac proj}
	\cite{Li2009}: For all $\alpha \in  (0,1)$, if $u \in H^1([0,T])$ such that $u(0)=0$, and $v \in H^{\alpha/2}([0,T])$, then $( \prescript{}{0}{ \mathcal{D}}_{t}^{\,\,\alpha} u, v )_{\Omega} =  (\, \prescript{}{0}{ \mathcal{D}}_{t}^{\,\,\alpha/2} u \,,\, \prescript{}{t}{ \mathcal{D}}_{T}^{\,\,\alpha/2} v\, )_{\Omega}$, where $(\cdot , \cdot)_{\Omega}$ represents the standard inner product in $\Omega=[0,T]$. 
\end{lem}
%
%\vspace{0.2 cm}
%
\begin{lem}
	\label{Lem: fractional integ-by-parts 1 and 2}
	\cite{kharazmi2017petrov}: Let $1 < \beta < 2$, $a$ and $b$ be arbitrary finite or infinite real numbers. Assume $u \in H^{\beta}(a,b)$ such that $u(a)=0$, also $\prescript{}{x}{\mathcal{D}}_{b}^{\beta/2}v$ is integrable in $\Omega = (a,b)$ such that $v(b) = 0$. Then, $( \prescript{}{a}{\mathcal{D}}_{x}^{\beta} u \,,\,v )_{\Omega} = ( \prescript{}{a}{\mathcal{D}}_{x}^{\beta/2} u \,,\,\prescript{}{x}{\mathcal{D}}_{b}^{\beta/2} v )_{\Omega}$.
\end{lem}
We generalize Lemma \ref{Lem: fractional integ-by-parts 1 and 2} to the two-sided $(1+d)$-dimensional case (see Appendix \ref{Sec: App. proof of gen int by part} for proof).
\begin{lem}
	\label{lem_generalize}
	Let $1<\beta_j<2$ for $j=1,2,\cdots,d$, and $u,v \in  \mathcal{X}_d$. Then,  
	\begin{align*}
	\big(\prescript{}{a_j}{\mathcal{D}}_{x_j}^{\beta_j} u, v\big)_{\Lambda_d}=\big(\prescript{}{a_j}{\mathcal{D}}_{x_j}^{\frac{\beta_j}{2}} u, \prescript{}{x_j}{\mathcal{D}}_{b_j}^{\frac{\beta_j}{2}} v\big)_{\Lambda_d},
	\qquad
	\big(\prescript{}{x_j}{\mathcal{D}}_{b_j}^{\beta_j} u, v\big)_{\Lambda_d}=\big(\prescript{}{x_j}{\mathcal{D}}_{b_j}^{\frac{\beta_j}{2}} u, \prescript{}{a_j}{\mathcal{D}}_{x_j}^{\frac{\beta_j}{2}} v\big)_{\Lambda_d}.
	\end{align*}
	%
	%Besides, if $1<2\mu_i<2$, and $u,v \in  \mathcal{X}_d$, then $\big(\prescript{}{x_i}{\mathcal{D}}_{b_i}^{2\mu_i} u, v\big)=\big(\prescript{}{x_i}{\mathcal{D}}_{b_i}^{\mu_i} u, \prescript{}{a_i}{\mathcal{D}}_{x_i}^{\mu_i} v\big),$ and
	%$\big(\prescript{}{a_i}{\mathcal{D}}_{x_i}^{2\mu_i} u, v\big)=\big(\prescript{}{a_i}{\mathcal{D}}_{x_i}^{\mu_i} u, \prescript{}{x_i}{\mathcal{D}}_{b_i}^{\mu_i} v\big).$ 
	%
\end{lem}
%

%%
%
%%%%%%%%%%%%%%%%%%%%%%%%%%%
\section{Problem Definition}
\label{Sec: Problem Definition}
%%%%%%%%%%%%%%%%%%%%%%%%%%%
%
Let $\Omega = (0,T] \times (a_1,b_1) \times (a_2,b_2) \times \cdots \times (a_d,b_d)$ be the computational domain for some positive integer $d$. We define $u(t,\textbf{x};\textbf{q}): \Omega \times Q \rightarrow \mathbb{R}$, where $\textbf{q} = \{ \alpha, \beta_1, \beta_2 , \cdots , \beta_d , k_1, k_2, \cdots , k_d \}$ is the vector of model parameters containing the fractional indices and model coefficients, and $Q = [0,1] \times [1,2]^d \times \mathbb{R}_{+}^{d}$ is the space of parameters. Thus, for any $\textbf{q} \in Q$, the transport field $u(t,\textbf{x};\textbf{q}): \Omega \rightarrow \mathbb{R}$. We consider the FPDE of strong form $\mathcal{L}^q(u) = f$, subject to Dirichlet initial and boundary conditions, where $\mathcal{L}$ is a linear two-sided fractional operator, given as follows
\begin{align}
	\label{Eq: FPDE}
	\prescript{}{0}{\mathcal{D}}_{t}^{\alpha} u(t,\textbf{x};\textbf{q}) 
	& - \sum_{j=1}^{d} \, k_{j} \,\left[ \prescript{}{a_j}{\mathcal{D}}_{x_j}^{\beta_j}
	+ \prescript{}{x_j}{\mathcal{D}}_{b_j}^{\beta_j} \right]
	u(t,\textbf{x};\textbf{q}) 
	= f(t,\textbf{x};\textbf{q}) ,
	\\
	\label{Eq: IC}
	& u \arrowvert_{t=0} = 0 ,  
	\\
	\label{Eq: BC}
	& u \arrowvert_{x=a_j} = u \arrowvert_{x=b_j} = 0 , 
\end{align}
in which $\alpha \in (0,1)$, $\beta_j \in (1,2)$, $k_{j}$ are real positive constant coefficients, and the fractional derivatives are taken in the Riemann-Liouville sense.

%
%%%%%%%%%%%%%%%%%%%%%%%%%%%%%
\subsection{\textbf{Weak Formulation}}
\label{Sec: FPDE Weak System}
%%%%%%%%%%%%%%%%%%%%%%%%%%%%%
%

For any set of model parameter $q$, we obtain the weak system, i.e. the variational form of the problem \eqref{Eq: FPDE} subject to the given initial/boundary conditions, by multiplying the equation with proper test functions and integrate over the whole computational domain $\Omega$. Therefore, using Lemmas \ref{Lem: left frac proj}-\ref{lem_generalize}, the bilinear form can be written as
\begin{align}
%
%\label{Eq: general weak form_2}
\label{Eq: bilinear form}
a(u,v)
=(\prescript{}{0}{\mathcal{D}}_{t}^{\frac{\alpha}{2}}\, u, \prescript{}{t}{\mathcal{D}}_{T}^{\frac{\alpha}{2}}\, v )_{\Omega} 
-\sum_{j=1}^{d} 
k_{j} \Big[ ( \prescript{}{a_j}{\mathcal{D}}_{x_j}^{\frac{\beta_j}{2}}\, u,\, \prescript{}{x_j}{\mathcal{D}}_{b_j}^{\frac{\beta_j}{2}}\, v )_{\Omega}
+ ( \prescript{}{x_j}{\mathcal{D}}_{b_j}^{\frac{\beta_j}{2}}\, u , \, \prescript{}{a_j}{\mathcal{D}}_{x_j}^{\frac{\beta_j}{2}} v)_{\Omega}
\Big] ,
\end{align}
and thus, by letting $\tilde{U}$ and $\tilde{V}$ be the proper solution/test spaces, the problem reads as: find $u \in \tilde{U}$ such that
\begin{align}
\label{Eq: general weak form FPDE}
a(u,v) = (f,v)_{\Omega}, \quad \forall v \in \tilde{V} .
\end{align}
%

%%
%
%%%%%%%%%%%%%%%%%%%%%%%%%%%
\section{Fractional Sensitivity Equation Method (FSEM)}
\label{Sec: FSEM}
%%%%%%%%%%%%%%%%%%%%%%%%%%%
%
We define the sensitivity coefficients as the partial derivative of transport field $u$ with respect to the model parameters $q_i$, i.e. 
\begin{align}
S_{u , q_i} = \frac{\partial \, u}{\partial \, q_i}, \quad i=1,2,\cdots,2d+1,
\end{align}
assuming that the partial derivative is well-defined. 
%We obtain the governing equation of evolution of sensitivity fields by taking the partial derivative of \eqref{Eq: FPDE}-\eqref{Eq: BC} with respect to each parameter $q_i$. 
To obtain the governing equation of evolution of sensitivity fields, i.e. FSEs, we first take the partial derivative of left- and right-sided fractional derivative \eqref{Eq: left RL derivative} and \eqref{Eq: right RL derivative} with respect to their orders. Therefore, by letting $\sigma \in (n-1,n]$, $x\in[a,b]$, $\mathcal{A}_n(\sigma) =\Gamma(n-\sigma) \frac{\partial}{\partial \sigma} \frac{1}{\Gamma(n-\sigma)} $, we have
\begin{align}
\label{Eq: Sensitivity of left sided fractional}
&\frac{\partial}{\partial \sigma} (\prescript{}{a}{\mathcal{D}}_{x}^{\sigma} u) = 
\prescript{}{a}{\mathcal{D}}_{x}^{\sigma} S_{u,\sigma}
+ \mathcal{A}_n(\sigma) \prescript{}{a}{\mathcal{D}}_{x}^{\sigma} u 
-\prescript{LP}{a}{\mathcal{D}}_{x}^{\sigma} u,
\\
\label{Eq: Sensitivity of right sided fractional}
&\frac{\partial}{\partial \sigma} (\prescript{}{x}{\mathcal{D}}_{b}^{\sigma} u) = 
\prescript{}{x}{\mathcal{D}}_{b}^{\sigma} S_{u,\sigma}
+ \mathcal{A}_n(\sigma) \prescript{}{x}{\mathcal{D}}_{b}^{\sigma} u 
- \prescript{LP}{x}{\mathcal{D}}_{b}^{\sigma} u.
\end{align}
The pre-super script LP stands for the \textit{Log-Pow integro-differential operator}, given in \eqref{Eq: left sided LP kernel derivative}-\eqref{Eq: right sided LP kernel derivative Caputo}, which we introduce here, for the first time in the context of FSEs. 

\begin{rem}
In the sequel, we only use the $RL-LP$ operator and thus, for the sake of simplicity, we drop the pre-super script $RL$ and $C$ and only use them when needed to distinguish between the two senses of derivatives.	
\end{rem}

%
%******************************************************************************************
%
\begin{figure}[t]
	\centering
	\centering
	\includegraphics[width=1\linewidth]{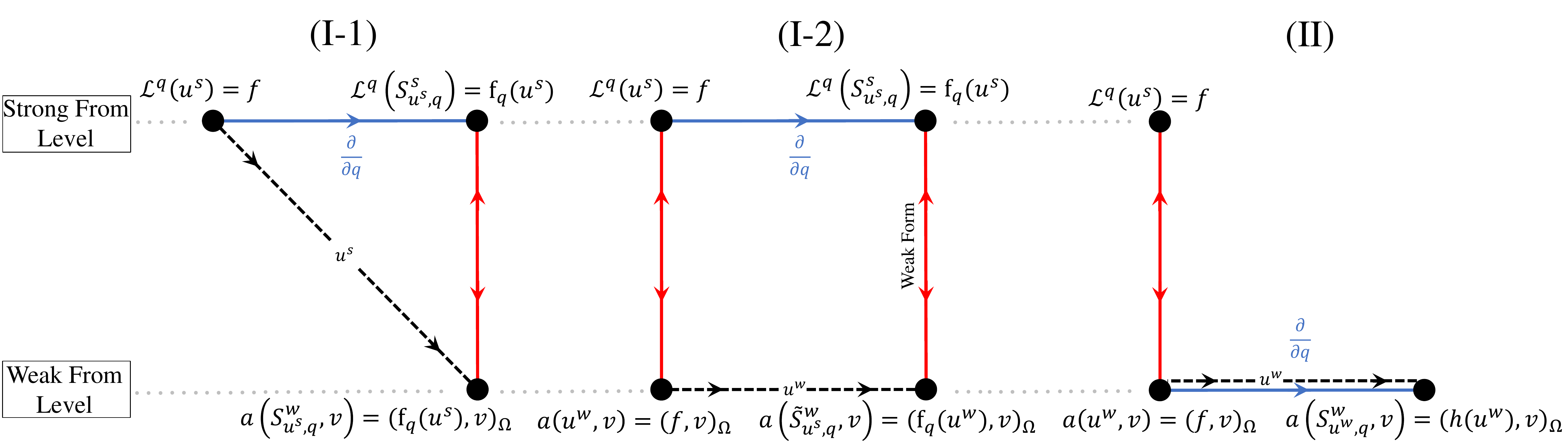}
	\caption{Schematic of strategies in deriving the weak form of FSEs. (I-1): first take $\frac{\partial}{\partial q}$ and then obtain the weak formulation, fed by strong solution $u^s$. (I-2): first take $\frac{\partial}{\partial q}$ and then obtain the weak formulation, fed by weak solution $u^w$. (II): first obtain the weak formulation and then take $\frac{\partial}{\partial q}$, fed by weak solution $u^w$.}
	\label{Fig: Schematic of FSE Derivationa}
\end{figure}
%
%****************************************************************************************** 
%

We derive the adjoint FSEs by pursuing two different strategies I and II, shown schematically in Fig. \ref{Fig: Schematic of FSE Derivationa}. We adopt the notation of $u^s$ and $u^w$ to distinguish the solution to strong and weak form of the problem for ease of describing the two following strategies. In the first strategy, we first take the partial derivative of FPDE with respect to the model parameters $q$, and then, obtain the weak form of problem. If $u^s$ is known, then we follow I-1 (left figure), otherwise we formulate and solve the weak form of FPDE to obtain weak solution $u^w$ and follow I-2 (middle figure). 
\begin{align}
\label{Eq: FSE I-1}
\text{I-1:} \quad
\mathcal{L}^q(u^s) = f \,\,
{\color{blue}\xrightarrow[]{\,\, \frac{\partial}{\partial q} \,\,}} \,\,
\mathcal{L}^q(S^s_{u^s,q}) = \text{f}_q(u^s) \,\, 
{\color{red}\xrightarrow[]{\text{weak form}}} \,\, 
a(S^w_{u^s,q} , v) = (\text{f}_q(u^s) , v)_{\Omega}
\\
\label{Eq: FSE I-2}
\text{I-2:} \quad
\mathcal{L}^q(u^s) = f \,\,
{\color{blue}\xrightarrow[]{\,\, \frac{\partial}{\partial q} \,\,}} \,\,
\mathcal{L}^q(S^s_{u^s,q}) = \text{f}_q(u^s) \,\, 
{\color{red}\xrightarrow[]{\text{weak form}}} \,\, 
a(\tilde{S}^w_{u^s,q} , v) = (\text{f}_q(u^w) , v)_{\Omega}
\end{align}
Via proper construction of the corresponding subspaces, we discretize and solve $a(S^w_{u^s,q} , v) = (\text{f}_q(u^s) , v)_{\Omega}$ and $ a(\tilde{S}^w_{u^s,q} , v) = (\text{f}_q(u^w) , v)_{\Omega} $ in I-1 and I-2, respectively. We can show that $\Vert \tilde{S}^w_{u^s,q}  - S^w_{u^s,q}  \Vert_{L^2} \rightarrow 0$ as $u^w \rightarrow u^s$ by stability/error analysis of employed numerical scheme, where the solution space has the extra regularity required by the Log-Pow integro-differential operator in $\text{f}_q$.
\begin{rem}
	The solution to strong form of FPDE, i.e. $u^s$ can be analytically/numerically computed (by Laplace transform and finite difference method for example), or may be available as prior experimental data, and thus, can be fed directly to construct $\text{\normalfont f}_q$ in FSEs (see left sub-figure in Fig. \ref{Fig: Schematic of FSE Derivationa}). This is used in parameter estimation for model construction, section \ref{Sec: FSEM Model Construction}. 
\end{rem}
In the second strategy, we first obtain the weak form of FPDE, and then take the partial derivative with respect to the model parameters $q$. In this case, we procure $(h(u^w),v)$ as the right hand side of weak formulation, which is fed by the weak solution $u^w$. In this case, the function $h$ requires less regularity for the solution space due to the Log-Pow integro-differential operator, since the order of kernel is less compare to the first strategy.
\begin{align}
\label{Eq: FSE II}
\text{II:} \quad
\mathcal{L}^q(u^s) = f \,\,
{\color{red}\xrightarrow[]{\text{weak form}}} \,\, 
a(u^w , v) = (f , v)_{\Omega} \,\,
{\color{blue}\xrightarrow[]{\,\, \frac{\partial}{\partial q} \,\,}} \,\,
a(S^w_{u^w,q} , v) = (h(u^w) , v)_{\Omega}
\end{align}

In the next subsection, we adopt the two strategies to derive adjoint FSE to a fractional initial value problem, where we show the corresponding right-hand-side and the imposed extra regularity in each case. We then, extend the derivation to the case FPDE, in which we adopt strategy I-2.

%
%%%%%%%%%%%%%%%%%%%%%%%%%%%%%
\subsection{\textbf{FSEM (FIVP)}}
\label{Sec: FSEM Adjoint FSE FIVP}
%%%%%%%%%%%%%%%%%%%%%%%%%%%%%
%

Let $\Omega = (0,T]$ be the computational time domain and define $u(t;\alpha): \Omega \times (0,1) \rightarrow \mathbb{R}$. We consider the case of fractional initial value problem (FIVP) by letting the coefficients $k_j$'s to be zero in \eqref{Eq: FPDE}, and thus obtain the following FIVP, subject to Dirichlet initial condition, as $\prescript{}{0}{\mathcal{D}}_{t}^{\alpha} u(t;\alpha) = f(t;\alpha)$, $u(0) = 0$. By taking the partial derivative with respect to $\alpha$, we obtain the adjoint FSE in the strong form as $\prescript{}{0}{\mathcal{D}}_{t}^{\alpha} S_{u,\alpha} = \text{f}_{\alpha}$, $S_{u,q}\vert_{(t=0)} = 0$, where $\text{f}_{\alpha} = S_{f,\alpha} - \mathcal{A}_1(\alpha)  \,\, \prescript{}{0}{\mathcal{D}}_{t}^{\alpha} u + \prescript{LP}{0}{\mathcal{D}}_{t}^{\alpha} u$. Following strategy I, we obtain 
\begin{align}
\label{Eq: weak FSE FIVP I}
&a({S}_{u,\alpha} , v)_{\Omega} = (\text{f}_{\alpha} , v)_{\Omega} ,
\\
&(\text{f}_{\alpha},v)_{\Omega}  =  
(S_{f,\alpha},v)_{\Omega} 
- \mathcal{A}_1(\alpha)  \,\, (\prescript{}{0}{\mathcal{D}}_{t}^{\frac{\alpha}{2}} u , \prescript{}{t}{\mathcal{D}}_{T}^{\frac{\alpha}{2}} v)_{\Omega} 
+ (\prescript{LP}{0}{\mathcal{D}}_{t}^{\alpha} u,v)_{\Omega}.
\end{align}
In this case, constructing the right-hand-side imposes extra strong regularity of $\Vert \prescript{LP}{0}{\mathcal{D}}_{t}^{\alpha} u \Vert_{L^2} < \infty$ to the solution of FIVP. However, by following startegy II, we obtain 
\begin{align}
\label{Eq: weak FSE FIVP II}
&a({S}_{u,\alpha} , v)_{\Omega} = h(v),
\\
&h(v) = 
(S_{f,\alpha},v)_{\Omega} 
+ (f,S_{v,\alpha})_{\Omega} 
- \mathcal{A}_1(\frac{\alpha}{2})  \,\, (\prescript{}{0}{\mathcal{D}}_{t}^{\frac{\alpha}{2}} u , \prescript{}{t}{\mathcal{D}}_{T}^{\frac{\alpha}{2}} v)_{\Omega} 
-(\prescript{}{0}{\mathcal{D}}_{t}^{\frac{\alpha}{2}} u , \prescript{}{t}{\mathcal{D}}_{T}^{\frac{\alpha}{2}} S_{v,\alpha})_{\Omega} 
\\
\nonumber
& \quad\quad
+\frac{1}{2} (\prescript{LP}{0}{\mathcal{D}}_{t}^{\frac{\alpha}{2}} u , \prescript{}{t}{\mathcal{D}}_{T}^{\frac{\alpha}{2}} v)_{\Omega} 
+\frac{1}{2} (\prescript{}{0}{\mathcal{D}}_{t}^{\frac{\alpha}{2}} u , \prescript{LP}{t}{\mathcal{D}}_{T}^{\frac{\alpha}{2}} v)_{\Omega},
\end{align}
where, the function $h$ imposes extra weak regularity of $\Vert \prescript{LP}{0}{\mathcal{D}}_{t}^{\frac{\alpha}{2}} u \Vert_{L^2} < \infty$ and $\Vert \prescript{LP}{t}{\mathcal{D}}_{T}^{\frac{\alpha}{2}} v \Vert_{L^2} < \infty$ to the solution. We computationally study and make sure that the solution to \eqref{Eq: weak FSE FIVP I} converges to \eqref{Eq: weak FSE FIVP II}.

%
%%%%%%%%%%%%%%%%%%%%%%%%%%%%%
\subsection{\textbf{FSEM (FPDE)}}
\label{Sec: FSEM Adjoint FSE FPDE}
%%%%%%%%%%%%%%%%%%%%%%%%%%%%%
%
We consider the problem \eqref{Eq: FPDE}-\eqref{Eq: BC}. We adopt strategy I-2 and derive the adjoint FSEs and their corresponding weak form, where to construct the right-hand-side, we also obtain the weak form of FPDE. Thus, we solve a coupled system of FPDE and FSEs. By taking the partial derivatives of \eqref{Eq: FPDE} with respect to model parameters $q_i, \, i=1,2,\cdots,2d+1$, we obtain the corresponding adjoint FSEs as
\begin{align}
\label{Eq: FSE}
\boxed{
\mathcal{L}^q \,\,  S_{u,\alpha} = \text{f}_{\alpha} \,\, , \quad
\mathcal{L}^q \,\, S_{u,\beta_j} = \text{f}_{\beta_j} \,\, , \quad
\mathcal{L}^q \,\, S_{u,k_j} = \text{f}_{k_j} \,\, , \quad j = 1,2,\cdots,d  \,\, , }
\end{align}
in which
\begin{align}
\label{Eq: Fractional Operator}
\mathcal{L}^q (\cdot) & = \prescript{}{0}{\mathcal{D}}_{t}^{\alpha} (\cdot)
- \sum_{j=1}^{d} \, k_{j} \,\left[ \prescript{}{-a_j}{\mathcal{D}}_{x_j}^{\beta_j}
+ \prescript{}{x_j}{\mathcal{D}}_{b_j}^{\beta_j} \right] (\cdot)
\\
\label{Eq: FSE-temporal force function}
\text{f}_{\alpha} & =
S_{f,\alpha}
- \mathcal{A}_1(\alpha) \prescript{}{0}{\mathcal{D}}_{t}^{\alpha} u 
+ \prescript{LP}{0}{\mathcal{D}}_{t}^{\alpha} u
\\
\text{f}_{\beta_j} & =
S_{f,\beta_j}
+ k_{j} \, \mathcal{A}_2(\beta_j) \,\left[  \prescript{}{a_k}{\mathcal{D}}_{x_j}^{\beta_j} + \prescript{}{x_j}{\mathcal{D}}_{b_k}^{\beta_j} \right] u
- k_{j} \, \left[  \prescript{LP}{a_k}{\mathcal{D}}_{x_j}^{\beta_j} + \prescript{LP}{x_j}{\mathcal{D}}_{b_k}^{\beta_j} \right] u,
%\quad j = 1,2,\cdots,d,
%
\\
\label{Eq: FSE-coeff force function}
\text{f}_{k_j} & =
	S_{f,k_j}
+ \,\left[ \prescript{}{-a_j}{\mathcal{D}}_{x_j}^{\beta_j}
+ \prescript{}{x_j}{\mathcal{D}}_{b_j}^{\beta_j} \right] u.
%\quad j = 1,2,\cdots,d.
%
\end{align}
Moreover, by taking the partial derivative of initial and boundary conditions \eqref{Eq: IC} and \eqref{Eq: BC}, respectively, with respect to model parameters, we obtain the following conditions for $i=1,2,\cdots,2d+1$, as
\begin{align}
\label{Eq: SE-ICBC}
S_{u,q_i} \big\vert_{t=0} = \frac{\partial S_{u,q_i} }{\partial t} \big\vert_{t=0} = 0 , \quad 
S_{u,q_i} \big\vert_{x=a_j} = S_{u,q_i} \big\vert_{x=b_j} = 0 , \,\, j=1,2,\cdots,d.
\end{align}
%

%
%%%%%%%%%%%%%%%%%%%%%%%%%%%%%
\subsection{\textbf{Mathematical Framework: Coupled System of The FPDE and Derived FSEs}}
\label{Sec: FSEM Mathematical Framework}
%%%%%%%%%%%%%%%%%%%%%%%%%%%%%
%
We extend the solution/test spaces, defined in \eqref{Eq: solution test space} by imposing the ``extra regularities" due to the right-hand-side of adjoint FSEs \eqref{Eq: FSE}, and define the proper underlying spaces for solving the coupled system of adjoint FSEs and FPDE.

%
%%%%%%%%%%%%%%%%%%%%%%%%%%%%%
\subsection*{\textbf{Solution/Test Spaces}}
\label{Sec: Sol Test Space FSE}
%%%%%%%%%%%%%%%%%%%%%%%%%%%%%
%
Let 
\begin{align*}
\mathcal{H}^{\frac{\beta_j}{2}}_{0}(\Lambda_j) = 
\Big\lbrace 
u \in H^{\frac{\beta_j}{2}}_{0}(\Lambda_j) 
\, \Big| \, 
\sqrt{\Vert \prescript{LP}{a_j}{\mathcal{D}}_{x_j}^{\beta_j} u \Vert^2_{L^2(\Lambda_j)} + \Vert \prescript{LP}{x_j}{\mathcal{D}}_{b_j}^{\beta_j} u \Vert^2_{L^2(\Lambda_j)}} < \infty 
\Big\rbrace, \,\, 
j=1,2,\cdots,d,
\end{align*}
associated with the norm $ \Vert \cdot \Vert_{{^c}H^{\frac{\beta_j}{2}}_{}(\Lambda_j)}$. We define $\mathcal{X}_1 = \mathcal{H}^{\frac{\beta_1}{2}}_{0}(\Lambda_1)$, and accordingly, $\mathcal{X}_j, \, j=2,\cdots,d$ as
\begin{align}
\mathcal{X}_2 &= \mathcal{H}^{\frac{\beta_2}{2}}_0 ((a_2,b_2); L^2(\Lambda_1) ) \cap L^2((a_2,b_2); \mathcal{X}_1),
\\
&\vdots
\nonumber
\\
\mathcal{X}_d &= \mathcal{H}^{\frac{\beta_d}{2}}_0 ((a_d,b_d); L^2(\Lambda_{d-1}) ) \cap L^2((a_d,b_d); \mathcal{X}_{d-1}),
\end{align}
associated with the similar norm $\Vert \cdot \Vert_{\mathcal{X}_d} $. Thus, we define the corresponding ``solution space" $\tilde{U}$ and ``test space" $\tilde{V}$, respectively, as
\begin{align}
\label{Eq: solution test space FSE}
\tilde{U}  = \prescript{l}{0}{\mathcal{H}}^{\frac{\alpha}{2}}\Big(I; L^2(\Lambda_d) \Big) \cap L^2(I; \mathcal{X}_d),
\quad
\tilde{V}  = \prescript{r}{0}{\mathcal{H}}^{\frac{\alpha}{2}}\Big(I; L^2(\Lambda_d)\Big) \cap L^2(I; \mathcal{X}_d),
\end{align}
endowed with similar norms \eqref{Eq: solution space norm} and \eqref{Eq: test space norm}, where 
\begin{align*}
\prescript{l}{0}{\mathcal{H}}^{\frac{\alpha}{2}} \Big(I; L^2(\Lambda_d) \Big) = 
\Big{\{} u \, \big|\, \Vert u(t,\cdot) \Vert_{L^2(\Lambda_d)} \in H^{\frac{\alpha}{2}}(I),
\Vert \prescript{LP}{0}{\mathcal{D}}_{t}^{\alpha} u \Vert_{L^2(I)} < \infty ,
u\vert_{t=0}=u\vert_{x=a_j}=u\vert_{x=b_j}=0,\, j=1,2,\cdots,d  \Big{\}},
\\
\prescript{r}{0}{\mathcal{H}}^{\frac{\alpha}{2}} \Big(I; L^2(\Lambda_d) \Big) = 
\Big{\{} v \,\big|\, \Vert v(t,\cdot) \Vert_{L^2(\Lambda_d)} \in H^{\frac{\alpha}{2}}(I), 
\Vert \prescript{LP}{t}{\mathcal{D}}_{T}^{\alpha} u \Vert_{L^2(I)} < \infty ,
v\vert_{t=T}=v\vert_{x=a_j}=v\vert_{x=b_j}=0,\, j =1,2,\cdots,d  \Big{\}},
\end{align*}
equipped with norms $\Vert u \Vert_{\prescript{l}{}H^{\frac{\alpha}{2}}(I; L^2(\Lambda_d))}$ and $\Vert u \Vert_{\prescript{r}{}H^{\frac{\alpha}{2}}(I; L^2(\Lambda_d))}$, respectively.

%
%%%%%%%%%%%%%%%%%%%%%%%%%%%%%
\subsection*{\textbf{Weak Formulation}}
\label{Sec: Weak System}
%%%%%%%%%%%%%%%%%%%%%%%%%%%%%
%
Since derived FSEs \eqref{Eq: FSE} preserve the structure of FPDE \eqref{Eq: FPDE}, the bilinear form of corresponding weak formulation takes the same form as \eqref{Eq: bilinear form}. Therefore, By letting $\tilde{U}$ and $\tilde{V}$ be the solution/test spaces, defined in \eqref{Eq: solution test space FSE}, the problem reads as: find $u \in \tilde{U}$ such that
\begin{align}
\label{Eq: general weak form FPDE coupled}
a(u,v) = (f,v)_{\Omega}, \quad \forall v \in \tilde{V} ,
\end{align}
and find $S_{u,q_i} \in U, \, 1=1,2,\cdots,2d+1$ such that
\begin{align}
\label{Eq: general weak form FSE}
a(S_{u,q_i},w) = (\text{f}_{q_i},w)_{\Omega} \quad \forall w \in V ,
\end{align}
where $U$ and $V$ are defined in \eqref{Eq: solution test space}.

%
%%%%%%%%%%%%%%%%%%%%%%%%%%%%%%%%%%%%%%
%
\subsection{\textbf{Petrov-Galerkin Spectral Method}}
\label{Sec: Implementation}
%
%%%%%%%%%%%%%%%%%%%%%%%%%%%%%%%%%%%%%%
%
We define the following finite dimensional solution and test spaces. We employ Legendre polynomials $\phi_{m_j}(\xi), \, j=1,2,\cdots,d$, and Jacobi poly-fractonomial of first kind $\psi^{\tau}_n(\eta)$ \cite{zayernouri2015tempered,Zayernouri2013}, as the spatial and temporal bases, respectively, given in their corresponding standard domain as
\begin{align}
\label{Eq: Spatial Basis}
\phi^{}_{m_j} ( \xi )  & =  \sigma_{m_j} \big{(} P_{m_j+1} (\xi) - P_{m_j-1} (\xi)\big{)},  \quad  \xi \in [-1,1]  \qquad m_j=1,2,\cdots ,
\\
\label{Eq: Temporal Basis}
\psi^{\tau}_n(\eta) & = {\sigma}_{n} \prescript{(1)}{}{ \mathcal{P}}_{n}^{\,\,\tau}(\eta) = {\sigma}_{n} (1+\eta)^{\tau} P_{n-1}^{-\tau, \tau} (\eta), \quad \eta\in [-1,1]  \quad n=1,2,\cdots ,
\end{align}
in which $\sigma_{m_j} = 2 + (-1)^{m_j}$. Therefore, by performing affine mappings $\eta = 2\frac{t}{T}-1$ and $\xi = 2\frac{x-a_j}{b_j-a_j} -1$ from the computational domain to the standard domain, we construct the solution space $U_N$ as
\begin{align}
\label{Eq: Solution Space :PG}
U_N = 
span \, \Big\{ \,\,   
\Big( \psi^{\,\tau}_n \circ \eta \Big) ( t ) \,\,
\prod_{j=1}^{d} \Big( \phi^{}_{m_j} \circ \xi \Big)  (x_j) \,\,
: n = 1,2, \cdots, \mathcal{N}, \,\, m_j= 1,2, \cdots, \mathcal{M}_j
\,\, \Big\}.
\end{align}
We note that the choice of temporal and spatial basis functions naturally satisfy the initial and boundary conditions, respectively. The parameter $\tau$ in the temporal basis functions plays a role of fine tunning parameter, which can be chosen properly to capture the singularity of exact solution.

Moreover, we employ Legendre polynomials $\Phi_{r_j}(\xi), \, j=1,2,\cdots,d$, and Jacobi poly-fractonomial of second kind $\Psi^{\tau}_k(\eta)$, as the spatial and temporal test functions, respectively, given in their corresponding standard domain as
\begin{align}
\label{Eq: Spatial Test}
\Phi_{r_j} ( \xi )  & =  \widetilde{\sigma}_{r_j} \big{(} P_{r_j+1}^{} (\xi) - P_{r_j-1}^{} (\xi)\big{)},  \quad  \xi \in [-1,1]  \qquad r_j =1,2,\cdots ,
\\
\label{Eq: Temporal Test}
\Psi^{\tau}_k(\eta) & = \widetilde{\sigma}_{k} \prescript{(2)}{}{ \mathcal{P}}_{k}^{\,\,\tau}(\eta) = \widetilde{\sigma}_{k} (1-\eta)^{\tau}\, P_{k-1}^{\tau,-\tau} (\eta), \quad \eta\in [-1,1]  \quad k=1,2,\cdots,
\end{align}
where $ \widetilde{\sigma}_{r_j} = 2\,(-1)^{r_j} + 1$. Therefore, by similar affine mapping we construct the test space $V_N$ as
\begin{align}
\label{Eq: Test Space: PG}
V_N = span \, \Big\{  \,\,
\Big(\Psi^{\tau}_k \circ \eta\Big)(t) \,\,
\prod_{j=1}^{d} \Big( \Phi^{}_{r_j} \circ \xi_j\Big)(x_j) \,\,
: k = 1,2, \cdots, \mathcal{N}, \,\, r_j= 1,2, \cdots, \mathcal{M}_j
\,\, \Big\}.
\end{align}
%
%We can also show that our choice of test functions belongs to the defined underlying test space $\tilde{V}$, defined in \eqref{Eq: solution test space FSE}. 
We can show that our choice of basis/test functions satisfy the extra regularity imposed by the Log-Pow integro-differential operator.
Thus, since $U_N \subset \tilde{U} \subset U$ and $V_N \subset \tilde{V} \subset V$, the problems \eqref{Eq: general weak form FPDE coupled} and \eqref{Eq: general weak form FSE} read as: find $u_N \in U_N$ such that
\begin{align}
\label{Eq: PG method FPDE}
a_h(u_N,v_N) = l(v_N), \quad \forall v_N \in V_N,
\end{align}
where $l(v_N) = (f,v_N)$; and find $Su_N \in U_N$ such that
\begin{align}
\label{Eq: PG method}
a_h(Su_N,w_N) = l(w_N), \quad \forall w_N \in V_N,
\end{align}
where $l(w_N) = (\text{f}_{q_i},w_N)$. Also, the discrete bilinear form $a_h(u_N,v_N)$ can be written as 
\begin{align}
\label{Eq: discrete weak form}
a_h(u_N,v_N)
=(\prescript{}{0}{\mathcal{D}}_{t}^{\frac{\alpha}{2}}\, u_N, \prescript{}{t}{\mathcal{D}}_{T}^{\frac{\alpha}{2}}\, v_N )_{\Omega} 
-\sum_{j=1}^{d} 
k_{j} \Big[ ( \prescript{}{a_j}{\mathcal{D}}_{x_j}^{\frac{\beta_j}{2}}\, u_N ,\, \prescript{}{x_j}{\mathcal{D}}_{b_j}^{\frac{\beta_j}{2}}\, v_N )_{\Omega}
+ ( \prescript{}{x_j}{\mathcal{D}}_{b_j}^{\frac{\beta_j}{2}}\, u_N , \, \prescript{}{a_j}{\mathcal{D}}_{x_j}^{\frac{\beta_j}{2}} v_N)_{\Omega}
\Big].
\end{align}
We expand the approximate solution $u_N \in U_N$, satisfying the discrete bilinear form \eqref{Eq: discrete weak form}, in the following form
\begin{align}
\label{Eq: PG expansion}
u_{N}(t,\textbf{x}) = 
\sum_{n=1}^\mathcal{N}
\sum_{m_1=1}^{\mathcal{M}_1}
\cdots 
\sum_{m_d= 1}^{\mathcal{M}_d}  \,\,
\hat u_{ n,m_1,\cdots,m_d} \,\,
\Big[\psi^{\tau}_n(t)
\prod_{j=1}^{d} \phi^{}_{m_j}(x_j)
\Big] ,
\end{align}
and obtain the corresponding Lyapunov system by substituting \eqref{Eq: PG expansion} into \eqref{Eq: discrete weak form} by choosing $v_N(t,\textbf{x}) = \Psi^{\tau}_k(t) \prod_{j=1}^{d} \Phi^{}_{r_j}(x_j)$, $k = 1,2, \dots, \mathcal{N}$, $r_j= 1,2, \dots, \mathcal{M}_j$. Therefore, 
\begin{align}
\label{Eq: general Lyapunov}
\Big[
S_{T} \otimes M_1 \otimes M_2 \cdots \otimes M_d 
&+
\sum_{j=1}^{d} 
M_{T} \otimes M_1\otimes \cdots   \otimes M_{j-1} \otimes S_{j}^{{Tot}} \otimes M_{j+1}  \cdots \otimes M_d
\nonumber
\\
&+ 
\gamma \, M_{T}\otimes M_1 \otimes M_2 \cdots \otimes M_d 
\Big] \, 
\mathcal{U}= F,
\end{align} 
in which $\otimes$ represents the Kronecker product, $F$ denotes the multi-dimensional load matrix whose entries are given as
\begin{eqnarray}
\label{Eq: general load matrix}
F_{k,r_1,\cdots, r_d} = \int_{\Omega}^{} f(t,\textbf{x}) \,
\Big(
\Psi^{\,\tau}_k \circ \eta \Big)(t)
\prod_{j=1}^{d} \Big(\Phi^{}_{r_j} \circ \xi_j\Big)(x_j)\, 
d\Omega,
\end{eqnarray}
and $\mathcal{U}$ is the matrix of unknown coefficients. The matrices $S_{T}$ and $M_{T}$ denote the temporal stiffness and mass matrices, respectively; and the matrices $S_{j}$ and $M_j$ denote the spatial stiffness and mass matrices, respectively. We obtain the entries of spatial mass matrix $M_j$ analytically and employ proper quadrature rules to accurately compute the entries of other matrices $S_{T}$, $M_{T}$ and $S_{j}$.

We note that the choices of basis/test functions, employed in developing the PG scheme leads to symmetric mass and stiffness matrices, providing useful properties to further develop a fast solver. The following Theorem \ref{Thm: fast solver} provides a unified fast solver, developed in terms of the generalized eigensolutions in order to obtain a closed-form solution to the Lyapunov system \eqref{Eq: general Lyapunov}.

\begin{thm}[Unified Fast FPDE Solver \cite{Samiee2016unified,samiee2017fast}]
	\label{Thm: fast solver}
	Let $\{ {\vec{e}_{}}^{\,\,\mu_j}    ,    \lambda^{}_{m_j}\,  \}_{m_j=1}^{\mathcal{M}_j}$ be the set of general eigen-solutions of the spatial stiffness matrix $S^{Tot}_j$ with respect to the mass matrix $M_{j}$. Moreover, let $\{ {\vec{e}_{n}}^{\,\,\tau}    ,    \lambda^{\tau}_{n}\,  \}_{n=1}^{\mathcal{N}}$ be the set of general eigen-solutions of the temporal mass matrix $M_{T}$ with respect to the stiffness matrix $S_{T}$. Then, the matrix of unknown coefficients $\mathcal{U}$ is explicitly obtained as
	\begin{equation}
	\label{Eq: thm u expression in terms of k}
	\mathcal{U} = 
	\sum_{n=1}^{\mathcal{N}}
	\,\,
	\sum_{m_1= 1}^{\mathcal{M}_1}
	\cdots 
	\sum_{m_d= 1}^{\mathcal{M}_d}
	\kappa_{ n,m_1,\cdots,\,m_d  } \,
	\,\vec{e}_n^{\,\,\tau}\,
	\otimes
	\,{\vec{e}_{m_1}}^{}\,\,
	\otimes
	\cdots
	\otimes
	\,{\vec{e}_{m_d}}^{},
	\end{equation}
	where $\kappa_{ n,m_1,\cdots,\,m_d }$ is given by 
	\begin{eqnarray}
	\label{Eq: thm k fraction_1}
	\kappa_{ n,m_1,\cdots,\,m_d  } =  \frac{(\,\vec{e}_n^{\,\,\tau}
		\,{\vec{e}_{m_1}}^{}
		\cdots
		\,{\vec{e}_{m_d}}^{}) F}
	{
		\Big[
		(\vec{e}_n^{\,\,\tau^T} \, S_{T} \, \vec{e}_n^{\,\,\tau}) \,
		\prod_{j=1}^{d} (\vec{e}_{m_j}^{T}  \,  M_{j} \,  {\vec{e}_{m_j}}^{}) \,
		\Big]
		\Lambda_{n,m_1,\cdots,m_d}
	},
	\end{eqnarray}
	in which the numerator represents the standard multi-dimensional inner product, and $\Lambda_{n,m_1,\cdots,m_d}$ is obtained in terms of the eigenvalues of all mass matrices as
	\begin{eqnarray}
	\nonumber
	&\Lambda_{n,m_1,\cdots, m_d} = \Big[
	(1+\gamma\,\, 
	\lambda^{\tau}_n)
	+
	\lambda^{\tau}_n
	\sum_{j=1}^{d}
	(
	\lambda^{}_{m_j}
	)
	\Big].  &
	\end{eqnarray}
	
\end{thm}

%
%
%
%
%%%%%%%%%%%%%%%%%%%%%%%%%%%%%%%%%%%%%%
%
\subsection{\textbf{Stability Analysis}}
\label{Sec: Stability and Convergence of PG}
%
%%%%%%%%%%%%%%%%%%%%%%%%%%%%%%%%%%%%%%
%
We show the well-posedness of defined problem and prove the stability of proposed numerical scheme. 
\begin{lem}
	\label{norm_223}
	Let $\alpha \in (0,1)$, $\Omega=I \times \Lambda_d$, and $u\in \prescript{l}{0}H^{\alpha/2}(I; L^2(\Lambda_d))$. Then, 
	\begin{equation*}
	\big\vert \left( \prescript{}{0}{\mathcal{D}}_{t}^{\alpha/2} u, \prescript{}{t}{\mathcal{D}}_{T}^{\alpha/2} v \right)_{\Omega} \big\vert \equiv \Vert u \Vert_{\prescript{l}{}H^{\alpha/2}(I; L^2(\Lambda_d))} \, \Vert v \Vert_{\prescript{r}{}H^{\alpha/2}(I; L^2(\Lambda_d))},
	\quad 
	\forall v \in \prescript{r}{0}H^{\alpha/2}(I; L^2(\Lambda_d)).
	\end{equation*}
\end{lem}
\begin{proof}
	See Appendix \ref{Sec: App. proof of norm time} for proof.
\end{proof}

By equivalence of function spaces $\prescript{l}{}H^{\sigma}_0(\Lambda)$, $\prescript{r}{}H^{\sigma}_0(\Lambda)$, and $\prescript{c}{}H^{\sigma}_0(\Lambda)$ and also their associated norms $\Vert \cdot \Vert_{{^l}H^{\sigma}_{}(\Lambda)}$, $\Vert \cdot \Vert_{{^r}H^{\sigma}_{}(\Lambda)}$, and $\Vert \cdot \Vert_{{^c}H^{\sigma}_{}(\Lambda)}$; and also by following similar steps as in Lemma \ref{norm_223}, we can also prove that
\begin{align}
\label{equiv_space}
\vert \big(\prescript{}{a_d}{\mathcal{D}}_{x_d}^{\beta_d/2} u, \prescript{}{x_d}{\mathcal{D}}_{b_d}^{\beta_d/2} v\big)_{\Lambda_d} \vert \equiv  \vert u \vert_{\prescript{c}{}H^{\beta_d/2} \Big((a_d,b_d); L^2(\Lambda_{d-1}) \Big)} \, \vert v \vert_{\prescript{c}{}H^{\beta_d/2} \Big((a_d,b_d); L^2(\Lambda_{d-1}) \Big)},
\\
\label{equiv_space2}
\vert \big(\prescript{}{x_d}{\mathcal{D}}_{b_d}^{\beta_d/2} u, \prescript{}{a_d}{\mathcal{D}}_{x_d}^{\beta_d/2} v\big)_{\Lambda_d} \vert \equiv  \vert u \vert_{\prescript{c}{}H^{\beta_d/2} \Big((a_d,b_d); L^2(\Lambda_{d-1}) \Big)} \, \vert v \vert_{\prescript{c}{}H^{\beta_d/2} \Big((a_d,b_d); L^2(\Lambda_{d-1}) \Big)}.
\end{align} 
\begin{lem}[Continuity]
	\label{continuity_lem}
	The bilinear form \eqref{Eq: bilinear form} is continuous, i.e.,
	\begin{align}
	\label{continuity_eq}
	\forall u \in U, \,\,
	%\mathcal{B}^{\tau,\nu_1,\cdots,\nu_d} (\Omega) 
	\exists \, \beta > 0,  
	\quad \text{s.t.} \quad 
	\vert a(u,v)\vert \leq 
	\beta \,\,   \Vert u \Vert_{U}   \,\,     \Vert v \Vert_{V},
	%\beta \, \Vert u \Vert_{\mathcal{B}^{\tau,\nu_1,\cdots,\nu_d}(\Omega)}\Vert v \Vert_{\mathfrak{B}^{\tau,\nu_1,\cdots,\nu_d}(\Omega)} 
	\quad \forall v \in V.
	\end{align}
\end{lem}
\begin{proof}
	The proof directly concludes from \eqref{equiv_space}, \eqref{equiv_space2} and Lemma \ref{norm_223}.
	%With the aid of \eqref{equiv_space} and lemma \ref{norm_223}, we directly conclude \eqref{continuity_eq}.
\end{proof}	

\begin{thm}[Stability]
	\label{inf_sup_d_lem}
	The following inf-sup condition holds for the bilinear form \eqref{Eq: bilinear form}, i.e.,
	\begin{align}
	\label{Eq: inf sup-time_d_well}
	\underset{ 0 \neq u \in U}{\inf} \,\, \underset{ 0 \neq  v \in V}{\sup}
	\frac{\vert a(u , v)\vert}{ \,\, \Vert v\Vert_{V} \,\, \Vert u \Vert_{U} } \geq \beta > 0,
	%
	%\underset{0 \neq u \in \mathcal{B}^{\tau,\nu_1,\cdots,\nu_d} (\Omega)} {\inf} \,\,\underset{0 \neq v \in\mathfrak{B}^{\tau,\nu_1,\cdots,\nu_d} (\Omega)}{\sup}
	%\frac{\vert a(u , v)\vert}{\Vert v\Vert_{\mathfrak{B}^{\tau,\nu_1,\cdots,\nu_d}(\Omega)}\Vert u\Vert_{\mathcal{B}^{\tau,\nu_1,\cdots,\nu_d}}(\Omega)} \geq \beta > 0, 
	%
	\end{align}
	where $\Omega = I \times \Lambda_d$ and $\underset{u \in U}{\sup} \,\, \vert a(u , v)\vert>0$.
\end{thm}

\begin{proof}
	See Appendix \ref{Sec: App. proof of stability} for proof.
\end{proof}

\begin{thm}[well-posedness]
	\label{Thm: well-posedness_1D}
	For all $0<\alpha<1$, $\alpha \neq 1$, and  $1<\beta_j<2$, and $j=1,\cdots,d$, there exists a unique solution to \eqref{Eq: general weak form FPDE}, continuously dependent on  $f$, which belongs to the dual space of $U$.
\end{thm}
\begin{proof}
	Lemmas \ref{continuity_lem} (continuity) and \ref{inf_sup_d_lem} (stability) yield the well-posedness of weak form \eqref{Eq: general weak form FPDE} in (1+d)-dimension due to the generalized Babu\v{s}ka-Lax-Milgram theorem.
\end{proof}

Since the defined basis and test spaces are Hilbert spaces, and $U_N \subset U$ and $V_N \subset V$, we can prove that the developed Petrov-Gelerkin spectral method is stable and the following condition holds
\begin{align}
\label{Eq: inf sup-time}
\underset{0 \neq  u_N \in U_N}{\inf}\, \, \underset{0 \neq  v \in V_N}{\sup}
\frac{\vert a(u_N , v_N)\vert}{\Vert v_N\Vert_{V} \,\, \Vert u_N\Vert_{U}} \geq \beta > 0, 
\end{align}
with $\beta > 0$ and independent of $N$, where $\underset{u_N \in U_N}{\sup} \vert a(u_N , v_N)\vert>0 , \,\, \forall v_N \in V_N$.

We recall again here that the adjoint FSEs have similar bilinear form; and since $\tilde{U} \subset U$ and $\tilde{V} \subset V$, the obtained results are also applicable to them.

%%
%\newpage

%
%%%%%%%%%%%%%%%%%%%%%%%%%%%
\section{Fractional Model Construction}
\label{Sec: FSEM Model Construction}
%%%%%%%%%%%%%%%%%%%%%%%%%%%
%
We employ the developed FSEM in order to construct an iterative algorithm to estimate model parameters from known solution (or available sets of data). We formulate the iterative algorithm by minimizing an objective model error function. We recall again here that in our fractional model, the set of model parameters is $q = \{ \alpha, \beta_1, \beta_2 , \cdots , \beta_d , k_1, k_2, \cdots , k_d \}$, and here, we mainly focus on estimation of fractional indices. Thus, assuming the model coefficients $\{k_1, k_2, \cdots , k_d\}$ to be given/known, we reduce the model parameter set to $q = \{ \alpha, \beta_1, \beta_2 , \cdots , \beta_d \} \in Q \subset \mathbb{R}^{1+d}$.

%
%%%%%%%%%%%%%%%%%%%%%%%%%%%
\subsection{\textbf{Model Error}}
\label{Sec: Model Error}
%%%%%%%%%%%%%%%%%%%%%%%%%%%
%
The fractional model can be simply visualized as Fig. \ref{Fig: Fractional Model}, where $\mathcal{L}^{q} u = f$. We denote by the superscript $(^*)$ as the exact values of quantities. Therefore, $\mathcal{L}^{q^*} u^* = f^*$, where $u^*$, $f^*$ are the exact solution and force functions, respectively, and $q^*$ is the set of exact model parameters. Obviously, by choosing different values of model parameters (fractional indices), the fractional model observes the input differently, and thus, results in a different output. This leads to two types of \textit{model error}, namely, type-I and type-II, described as follows. We note that the introduced model errors are zero at the exact values $q^*$, by definition.

%
%******************************************************************************************
%
\begin{figure}[t]
	\centering
	\includegraphics[width=0.3\linewidth]{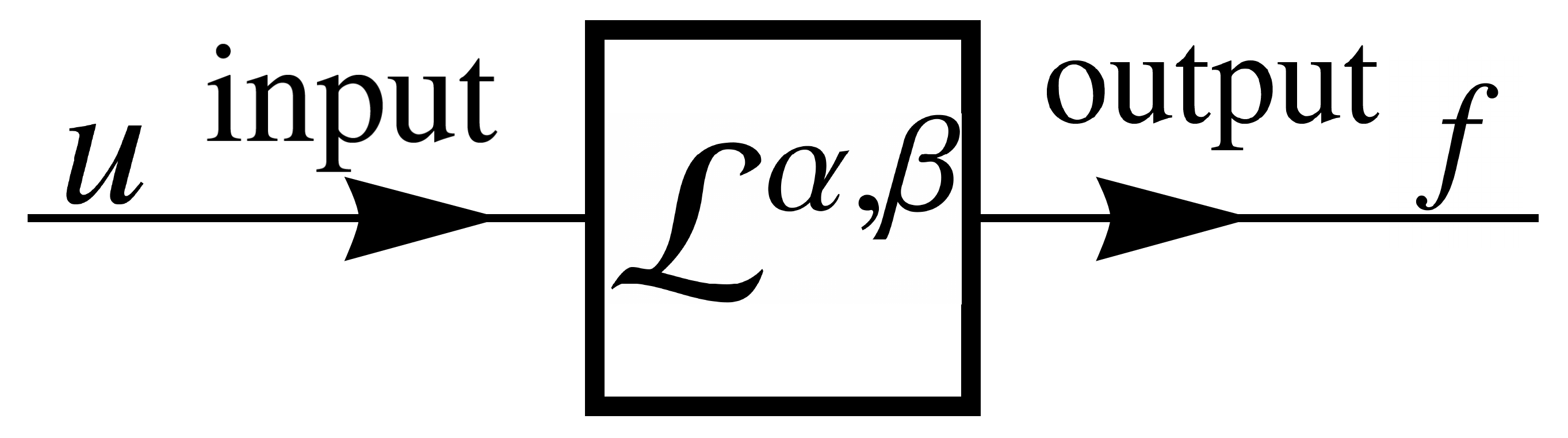}
	\vspace{-0.1 in}
	\caption{Schematic of fractional model}
	\label{Fig: Fractional Model}
\end{figure}
%
%****************************************************************************************** 
%
%
%******************************************************************************************
%
\begin{figure}[t]
	\centering
	\begin{subfigure}{0.3\textwidth}
		\centering
		\includegraphics[width=1\linewidth]{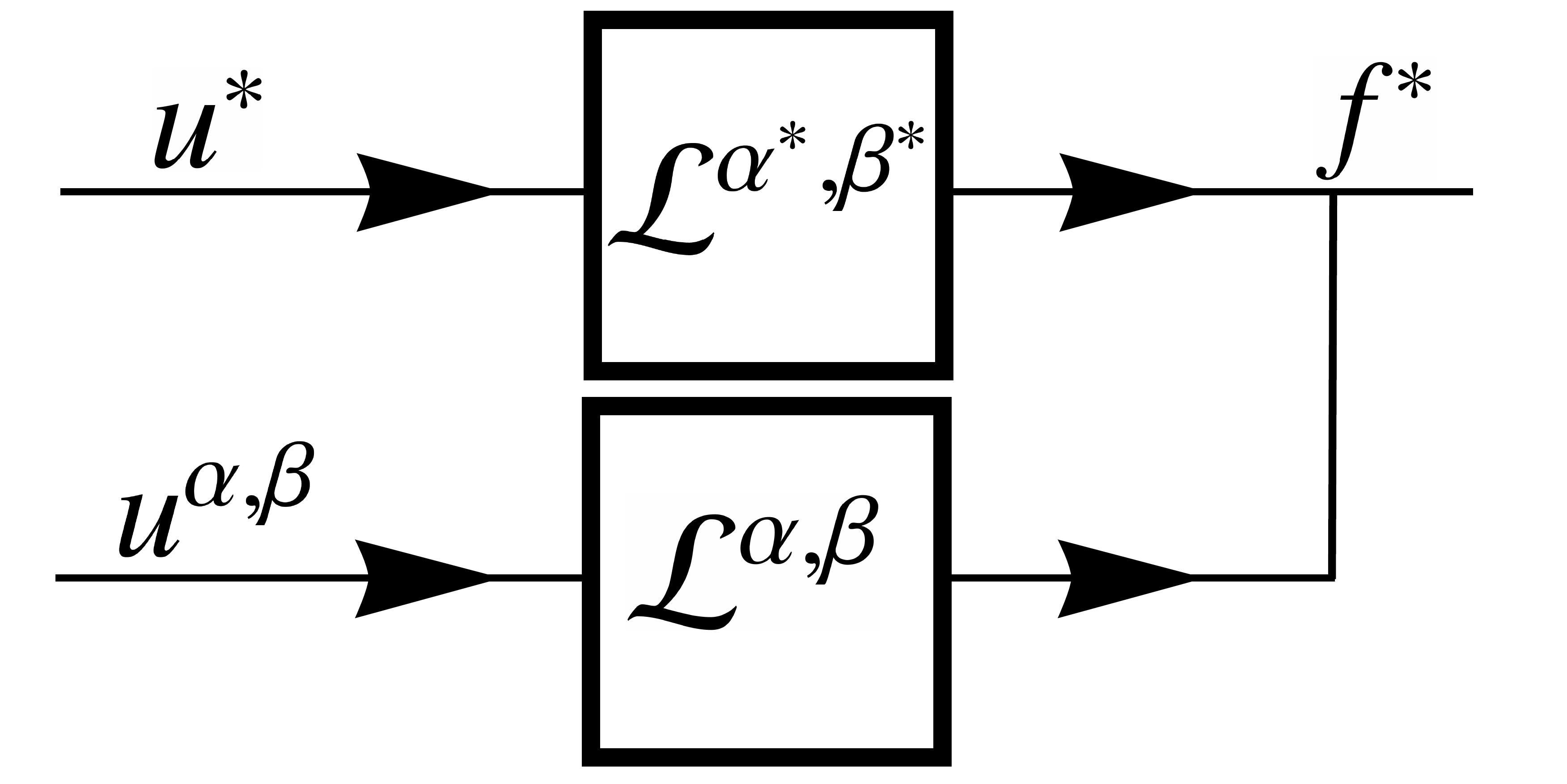}
		\caption{}
		%\label{Fig:}
	\end{subfigure}
	\begin{subfigure}{0.3\textwidth}
		\centering
		\includegraphics[width=1\linewidth]{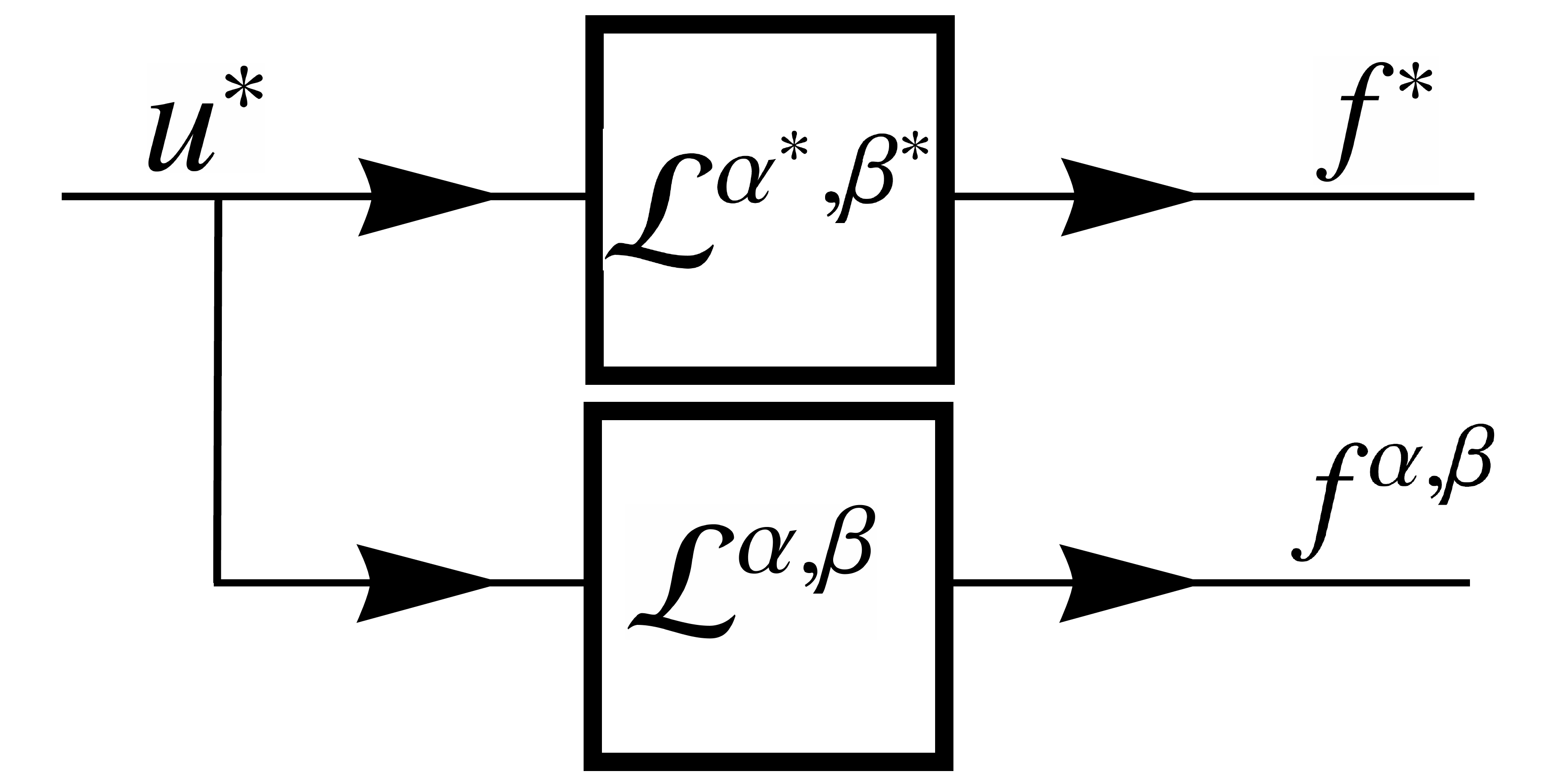}
		\caption{}
		%\label{Fig:}
	\end{subfigure}	
	\vspace{-0.25 in}
	\caption{Schematic of variation of fractional model based on (a) model error type-I and (b) model error type-II.}
	\label{Fig: Model Error}
\end{figure}
%
%****************************************************************************************** 
%

%\newpage
%\vspace{0.2in}
%
%%%%%%%%%%%%%%%%%%%%%%%%%%%
\subsubsection{Model Error: Type-I}
\label{Sec: Model Error Type I}
%%%%%%%%%%%%%%%%%%%%%%%%%%%
%
In model error type-I, we consider the output of model to be fixed, i.e, $f = f^*$, however, changing parameters makes the fractional model to observe the variated input $u^{q}$ as opposed to $u^{*}$. Therefore, we define the model error as the difference between variated and exact inputs, i.e. $E(q) = || u^{q} - u^* ||_{L^2}$. The schematic of variation of model from the exact model is shown in Fig. \ref{Fig: Model Error} (a). For each variated model, we accurately compute the numerical approximation, $u_N^{q}$, by solving \eqref{Eq: FPDE}, where by increasing the number of terms in the approximate solution, we make sure that the function $E(q) = || u_N^{q} - u^* ||_{L^2}$ solely describes the model error with minimum discretization error. The proposed iterative algorithm, as will be discussed later, involves the gradient of model error with respect to the model parameters. Thus, we take the partial derivative of $E$ with respect to $q$, as 
\begin{align}
\label{Eq: model error type I derivative}
S_{E,q}=\frac{\partial E}{\partial q} 
= \frac{ \int_{\Omega} S_{u^{q},q} \, (u_N^{q} - u^*) \,\, d\Omega}{E}
\end{align}
where ${S}_{u^{q},q}$ denotes the sensitivity fields, which is numerically obtained by solving FSEs \eqref{Eq: FSE}. We note that in this case, since $f$ is fixed and therefore, not sensitive to any parameter, we exclude the first term in the definition of force functions $\text{f}_{\alpha}$ and $\text{f}_{\beta}$.

%\vspace{0.2in}
%
%%%%%%%%%%%%%%%%%%%%%%%%%%%
\subsubsection{Model Error: type-II}
\label{Sec: Model Error Type II}
%%%%%%%%%%%%%%%%%%%%%%%%%%%
%
In model error type-II, we consider the input of model to be fixed, i.e, $u = u^*$, however, changing parameters makes the fractional model to result in the variated output $f^{q}$ as opposed to $f^{*}$. Therefore, we define the model error as the difference between variated and exact outputs, i.e. $E(q) = || f^{q} - f^* ||_{L^2}$. The schematic of variation of model from the exact model is shown in Fig. \ref{Fig: Model Error} (b). In this case, unlike model error type-I, the model error and its gradient can be expressed analytically. Therefore, they do not contain any  discretization error.

%
%%%%%%%%%%%%%%%%%%%%%%%%%%%
\subsection{\textbf{Model Error Minimization: Iterative Algorithm}}
\label{Sec: Iterative solver}
%%%%%%%%%%%%%%%%%%%%%%%%%%%
%

We minimize the model error by formulating a two-stages algorithm. Since we do not have prior information about the variated solution/force function, it is difficult to analytically predict the behavior of introduced model error. However, in every example, we numerically study the behavior of a low resolution model error manifold on a coarse grid, and then, perform the local minimization. The minimization problem is written as: 
\begin{align}
\underset{q \,  \in \, Q}{\min} \,  \Big( \, E(q) \, \Big)  ,
\end{align}
in which $E(q):Q \rightarrow \mathbb{R}$, and we assume that the problem is solvable, i.e. there exist a minimum point $q^* \in Q$. Proper choice of initial guess in local minimization is of great importance, where a wrong initial guess, not falling within small enough adjacency of minimum, may never converge. Therefore, the iterative convergence in a hypercube space of parameters is highly connected to an optimal initial guess for each parameter. In the sequel, we delineate the two stages of our algorithm, namely, stage I: \textit{coarse grid} searching, and stage II: \textit{nearby solution}.

In stage I, we progressively divide the hypercube parameter space into subspaces to narrow down the objective search region into a smaller region. This division process is not necessarily unique and can be done in different ways, among which we discuss the easy-to-implement one here, where in each progression step, we choose the subspace with minimum error at its corner. We carry out the coarse grid searching till we reach a small enough region, in which the nearby solution (stage II) is valid. As an example, we consider a $(1+1)$-D fractional model with $q^* = \{ \alpha^*, \beta^* \} =\{ 0.3, 0.8\}$ as the exact fractional indices in the parameter surface, shown in Fig. \ref{Fig: Coarse Grid Searching}. We divide the parameter space into four equal subspaces and by computing the error at corner points of each subsurface (black dots), we shrink the search region (to the labeled subsurface $3$). We progress further once again in a similar fashion, divide the subsurface, and compute the error at corner points (red dots). We finally, narrow down the search region into labeled subsurface $31$. We see that in this case, with computing the error only at 14 points, we can efficiently narrow down the parameter space into a small enough search region, in which we can perform stage II of the algorithm.

%
%******************************************************************************************
%
\begin{figure}[htb]
	\centering
	\begin{tikzpicture}
	\node (img)  {\includegraphics[scale=1]{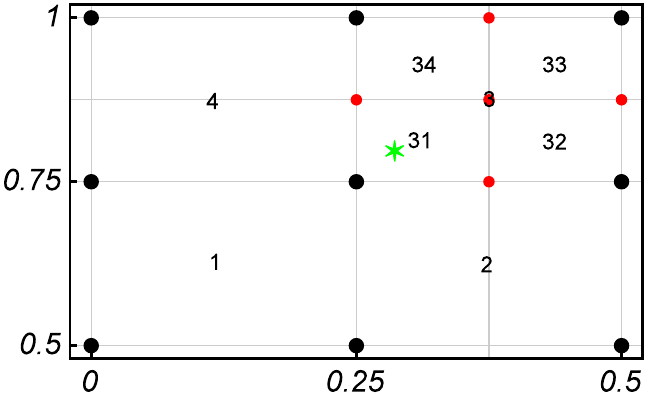}};
	\node[below=of img, node distance=0cm, xshift=0.3cm, yshift=1cm] {$\alpha$};
	\node[left=of img, node distance=0cm, rotate=90, xshift=0.4cm, yshift=-0.8cm] {$\beta$};
	\end{tikzpicture}
	\vspace{-0.15 in}
	\caption{Iterative algorithm: coarse grid searching for $(1+1)$-D parameter space, where $\alpha^* = 0.3$ and $\beta^* = 0.8$.}	
	\label{Fig: Coarse Grid Searching}
\end{figure}
%
%******************************************************************************************
%

In stage II of the algorithm, we employ a gradient decent method, in which by starting from an initial guess $q^0 =\{ \alpha^0 , \beta_1^0 , \beta_2^0, \cdots , \beta_d^0 \}$ in the obtained search region from stage I, we produce a minimizing sequence $q^i, \, i=1,2,\cdots$, where 
\begin{align}
q^{i+1} = q^{i} + \Delta q^{i} ,
\end{align}
and the increment $\Delta q^{i} = s^i p^i$ contains both the step size $s^i$ and normalized step direction $p^i$. The superscript $i$ indicates the iteration index. We obtain the normalized direction $p^i$ by computing the gradient of model error with respect to the parameters. The step size is usually computed by performing a line search such that $E(q^{i} + s \, p^{i})$ is minimized over $\forall s \in \mathbb{R}$. However, in our case the method does not produce well-scaled search directions, and we need to approximate the current step size, using the previous one. Thus,
\begin{align}
\label{Eq: direction and step sizes}
p^i = -\frac{\nabla E(q^{i})}{\Vert \nabla E(q^{i}) \Vert},
\qquad
s^{i} = s^{i-1} \frac{{\nabla E(q^{i-1})}^{T} \, p^{i-1}}{{\nabla E(q^{i})}^{T} \, p^{i}},
\end{align}
where the first iteration size is obtained, using the Taylor expansion of model error about $q^0$.

%\vspace{0.2in}
%
%%%%%%%%%%%%%%%%%%%%%%%%%%%
\subsection{Fractional Model Construction: FSEM-based Iterative Algorithm}
%\label{Sec: FSEM FDE}
%%%%%%%%%%%%%%%%%%%%%%%%%%%
%
Let $\Omega = [0,T] \times [-1,1]$ be the computational domain. We consider the $(1+1)$-D case of FPDE \eqref{Eq: FPDE}, subject to the initial and boundary conditions \eqref{Eq: IC} and \eqref{Eq: BC}, respectively, where the adjoint FSEs are given in \eqref{Eq: FSE}. Assuming that the exact transport field $u^{*}(t,x)$ and force function $f^*(t,x)$ are given, then,
\begin{align}
	\label{Eq: FPDE ext}
	\prescript{}{0}{\mathcal{D}}_{t}^{\alpha^{*}} u^{*} 
	- k \left( \, \prescript{}{-1}{\mathcal{D}}_{x}^{\beta^{*}} + \, \prescript{}{x}{\mathcal{D}}_{1}^{\beta^{*}} \right) u^{*}
	= f^* 
\end{align}
in which $\lbrace \alpha^{*} , \beta^{*} \rbrace$ are the exact fractional indices and the coefficient $k$ is known.

By considering the two types of model error, we use the developed iterative formulation and follow Algorithm \ref{FPDE model error I} and Algorithm \ref{FPDE model error II} to obtain the optimal model parameters. In each iteration, the increments are obtained, using \eqref{Eq: direction and step sizes}.

\begin{algorithm}
	\caption{Fractional Model Construction: FSEM based Iterative Algorithm (Model Error type-I)}
	\label{FPDE model error I}
	\begin{algorithmic}[1]
		%\Procedure{}{}
		\STATE{Initial guess $q^0 = \lbrace \alpha^0,\beta^0 \rbrace$ }
		\STATE \textbf{Do} $i=0,1,\cdots$
		\STATE \quad Solve for $u^{q^i}_N$: FPDE
		\STATE \quad Compute the model error $E = || u_N^{q^i} - u^* ||_{L^2}$
		\STATE \quad \quad \textbf{If} $E < \text{tolerance}$, \textbf{Then} Break, \textbf{Otherwise} Continue
		\STATE \quad Solve for sensitivity fields: FSEs
		\STATE \quad Compute the model error gradient using sensitivity field
		\STATE \quad Compute the iteration increment $\Delta q^i$
		\STATE \quad March in parameter space $q^{i+1} = q^i + \Delta q^i$
		\STATE \textbf{End}
		%\EndProcedure
	\end{algorithmic}
\end{algorithm}
\begin{algorithm}
	\caption{Fractional Model Construction: FSEM based Iterative Algorithm (Model Error type-II)}
	\label{FPDE model error II}
	\begin{algorithmic}[1]
	%\Procedure{}{}
	\STATE \quad Initial guess $q^0  = \lbrace \alpha^0,\beta^0 \rbrace$ 
	\STATE \textbf{Do} $i=0,1,\cdots$
	\STATE \quad Compute the model error $E = || \mathcal{L}^{q^i} u^* - \mathcal{L}^{q^*} u^* ||_{L^2}$ 
	\STATE \quad \quad \textbf{If} $E < \text{tolerance}$, \textbf{Then} Break, \textbf{Otherwise} Continue
	\STATE \quad Compute the model error gradient using sensitivity field (analytically available)
	\STATE \quad Compute the iteration increment $\Delta q^{i} $
	\STATE \quad March in parameter space $q^{i+1} = q^i + \Delta q^i$
	\STATE \textbf{End}
	%\EndProcedure
\end{algorithmic}
\end{algorithm}

\begin{rem}
	\label{Rem: Multi variable iteration increament}
	In the first iteration, we compute the step size, using the Taylor expansion of the model error about the initial guess ${\lbrace \alpha^0 , \beta^0 \rbrace}$, which we separate into two directions as
	%	The iterative scheme is performed such that in each iteration we add increment in the direction of either $\alpha$ or $\beta$ while considering the other to be fixed. Therefore, the Taylor expansion \eqref{Eq: two var Taylor Expansion of E - FPDE} can be separated into two directions as
	%
	\begin{align}
	\label{Eq: Taylor Expansion of E - FPDE}
	E \Big|_{\lbrace \alpha^{*} , \beta^{*} \rbrace} \approx E \Big|_{\lbrace \alpha^{0} , \beta^0 \rbrace} + S_{E,\alpha} \Big|_{\lbrace \alpha^{0} , \beta^0 \rbrace} \,\, (\alpha^{*} - \alpha^{0})  ,
	\qquad 
	E \Big|_{\lbrace \alpha^{*} , \beta^{*} \rbrace} \approx E \Big|_{\lbrace \alpha^{0} , \beta^0 \rbrace} + S_{E,\beta} \Big|_{\lbrace \alpha^{0} , \beta^0 \rbrace} \,\, (\beta^{*} - \beta^0) .
	\end{align}
	Knowing that $E \Big|_{\lbrace \alpha^{*} , \beta^{*} \rbrace} = 0$, we obtain the parameters at next iterations as $\alpha^{1} = \alpha^0 + \Delta \alpha^0$ and $\beta^{1} = \beta^0 + \Delta \beta^0$, in which
	\begin{align}
	\label{Eq: FPDE iteration}
	\Delta\alpha^{0} \approx \, -\frac{E \Big|_{\lbrace \alpha^{0} , \beta^0 \rbrace}}{S_{E,\alpha} \Big|_{\lbrace \alpha^{0} , \beta^0 \rbrace}} \,\, ,
	\qquad \qquad
	\Delta\beta^{0} \approx \,  -\frac{E \Big|_{\lbrace \alpha^{0} , \beta^0 \rbrace}}{S_{E,\beta} \Big|_{\lbrace \alpha^{0} , \beta^0 \rbrace}}.
	\end{align}
\end{rem}
\vspace{0.2in}

%%
%\newpage
%
%%%%%%%%%%%%%%%%%%%%%%%%%%%
\section{\textbf{Numerical Results}}
\label{Sec: numerical results part I}
%%%%%%%%%%%%%%%%%%%%%%%%%%%
%
In the first part of numerical results, we investigate the performance of developed PG scheme in solving FPDE and the adjoint FSEs. We consider the coupled $(1+1)$-d FPDE and FSEs with one-sided fractional derivative and $k=1$, as 
\begin{align}
\label{Eq: Numerical FPDE}
&
\prescript{}{0}{\mathcal{D}}_{t}^{\alpha} u 
- \prescript{}{-1}{\mathcal{D}}_{x}^{\beta} \, u 
= f ,
\\
\label{Eq: Numerical FSE temporal}
&
\prescript{}{0}{\mathcal{D}}_{t}^{\alpha} S_{u,\alpha} 
- \prescript{}{-1}{\mathcal{D}}_{x}^{\beta}	S_{u,\alpha} 
= S_{f,\alpha}
- \mathcal{A}_1(\alpha) \prescript{}{0}{\mathcal{D}}_{t}^{\alpha} u 
+ \prescript{LP}{0}{\mathcal{D}}_{t}^{\alpha} u,
\\ 
\label{Eq: Numerical FSE spatial}
& 
\prescript{}{0}{\mathcal{D}}_{t}^{\alpha} S_{u,\beta} 
- \prescript{}{-1}{\mathcal{D}}_{x}^{\beta} S_{u,\beta} 
= S_{f,\beta}
+ \mathcal{A}_2(\beta) \, \prescript{}{-1}{\mathcal{D}}_{x}^{\beta}  u
- \prescript{LP}{-1}{\mathcal{D}}_{x}^{\beta} u.
\end{align}
We consider two cases of exact solution as
\begin{itemize}
	\item Case I: $u^{ext}(t,x) =  t^{3+\alpha/2} \, \left((1+x)^{3+\beta/2}-\frac{1}{2}(1+x)^{4+\beta/2}\right)$,
	
	\item Case II: $u^{ext}(t,x) = t^{3+\alpha/2} (t-0.4)(t-0.9)  \,\, \left((1+x)^{3+\beta/2}-\frac{1}{2}(1+x)^{4+\beta/2}\right)$.
	
\end{itemize}
where $\alpha/2 = 0.25$, and $\beta/2 = 0.75$. The exact solution and sensitivity fields, obtained by taking $\frac{\partial}{\partial \alpha}$ and $\frac{\partial}{\partial \beta}$ of the exact solutions, are shown in Fig. \ref{Fig: uext suext Case I} and \ref{Fig: uext suext Case II} for the two cases I and II, respectively. We employ the developed PG method to solve FPDE \eqref{Eq: Numerical FPDE} and obtain $u_N$, which we use to construct the right hand side of adjoint FSEs. Then, we again employ the developed PG method to solve FSEs \eqref{Eq: Numerical FSE temporal} and \eqref{Eq: Numerical FSE spatial} and obtain the numerical sensitivity fields, $S_{N_{u,\alpha}} , \, S_{N_{u,\beta}}$. We study the $L^2$-norm convergence of our proposed method by increasing the number of basis functions, as shown in Fig. \ref{Fig: uext suext error case I}.

%
%******************************************************************************************
%
\begin{figure}[t]
	\centering
	\begin{subfigure}{0.32\textwidth}
		\centering
		\includegraphics[width=1\linewidth]{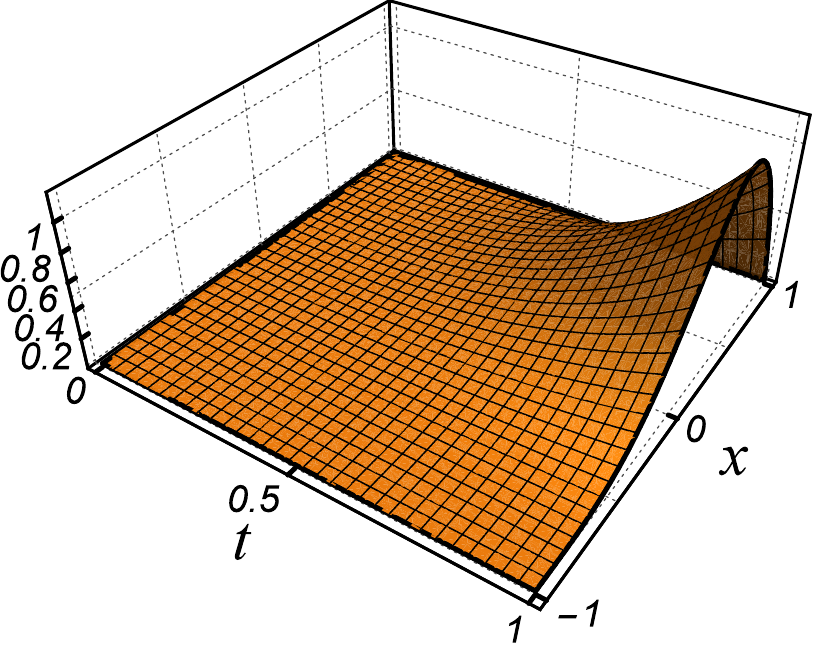}
%		\caption{}
		%\label{fig:}
	\end{subfigure}
	\begin{subfigure}{0.32\textwidth}
		\centering
		\includegraphics[width=1\linewidth]{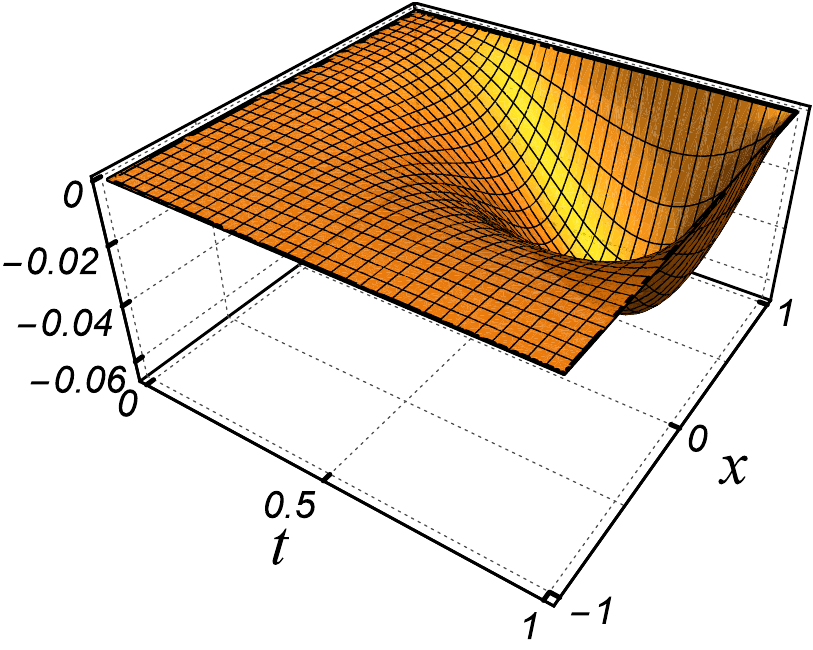}
%		\caption{}
		%\label{fig:}
	\end{subfigure}
	\begin{subfigure}{0.32\textwidth}
		\centering
		\includegraphics[width=1\linewidth]{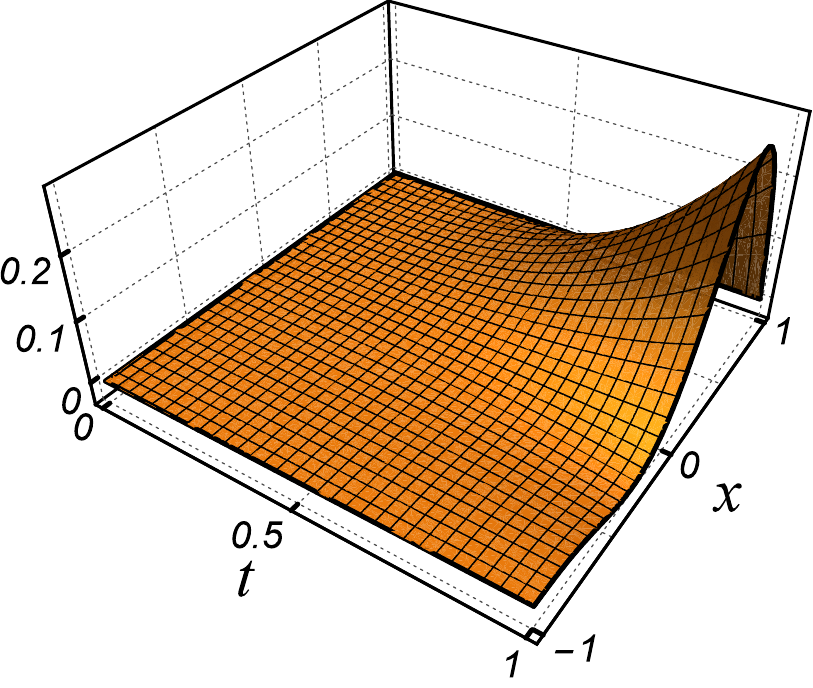}
%		\caption{}
		%\label{fig:}
	\end{subfigure}
	\vspace{-0.15 in}
	\caption{Plot of exact functions for case I with $\alpha/2 = 0.25$ and $\beta/2 = 0.75$: exact solution $u^{ext} $ (left), exact sensitivity field $S_{u^{ext},\alpha} = \frac{\partial u^{ext}}{\partial \alpha}$ (middle), exact sensitivity field $S_{u^{ext},\beta} = \frac{\partial u^{ext}}{\partial \beta}$ (right).}
	\label{Fig: uext suext Case I}
\end{figure}
%
%******************************************************************************************
%******************************************************************************************
%
\begin{figure}[t]
	\centering
	\begin{subfigure}{0.32\textwidth}
		\centering
		\includegraphics[width=1\linewidth]{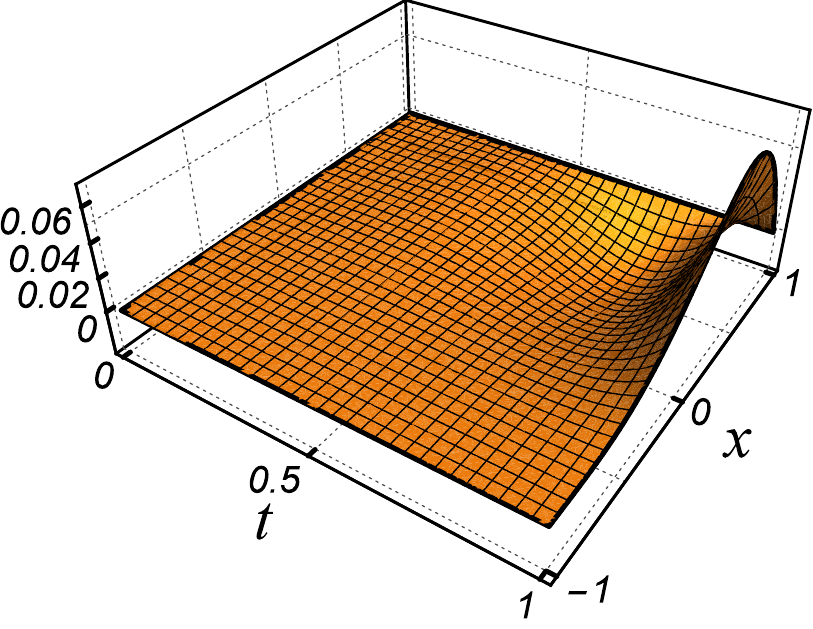}
%		\caption{}
		%\label{fig:}
	\end{subfigure}
	\begin{subfigure}{0.32\textwidth}
		\centering
		\includegraphics[width=1\linewidth]{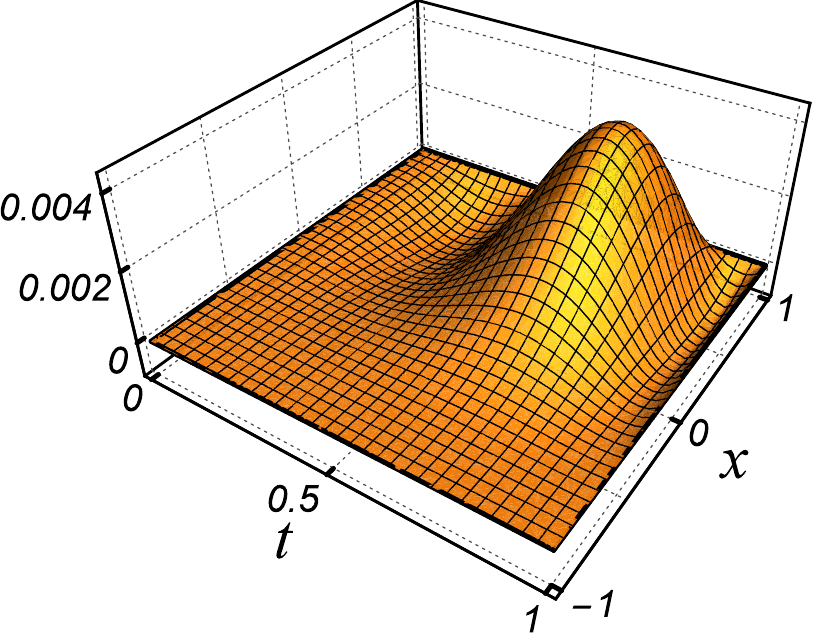}
%		\caption{}
		%\label{fig:}
	\end{subfigure}
	\begin{subfigure}{0.32\textwidth}
		\centering
		\includegraphics[width=1\linewidth]{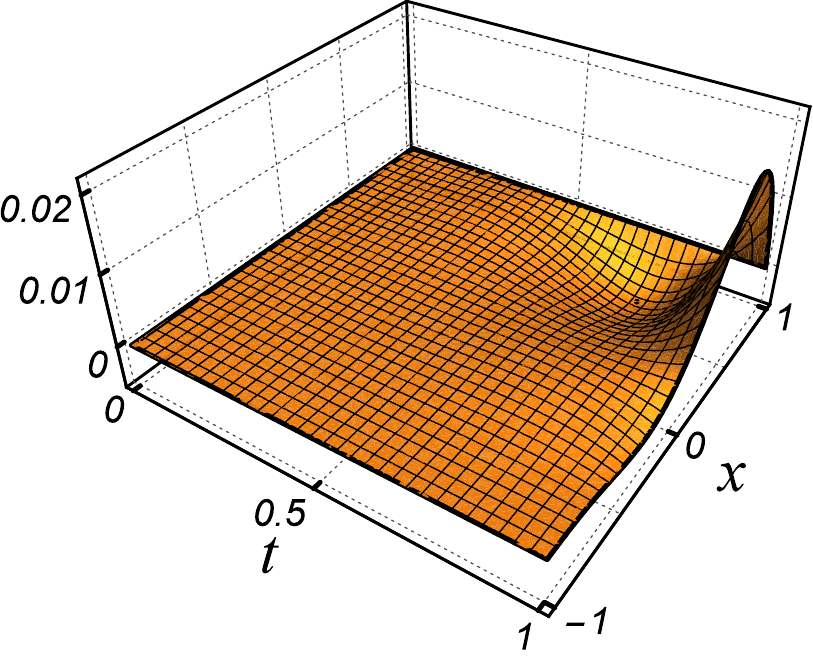}
%		\caption{}
		%\label{fig:}
	\end{subfigure}
	\vspace{-0.15 in}
	\caption{Plot of exact functions for case II with $\alpha/2 = 0.25$ and $\beta/2 = 0.75$: exact solution $u^{ext}$ (left), exact sensitivity field $S_{u^{ext},\alpha}$ (middle), exact sensitivity field $S_{u^{ext},\beta}$ (right).}
	\label{Fig: uext suext Case II}
\end{figure}
%
%******************************************************************************************

%******************************************************************************************
%
\begin{figure}[t]
	\centering
	\begin{subfigure}{0.49\textwidth}
		\centering
		\includegraphics[width=1\linewidth]{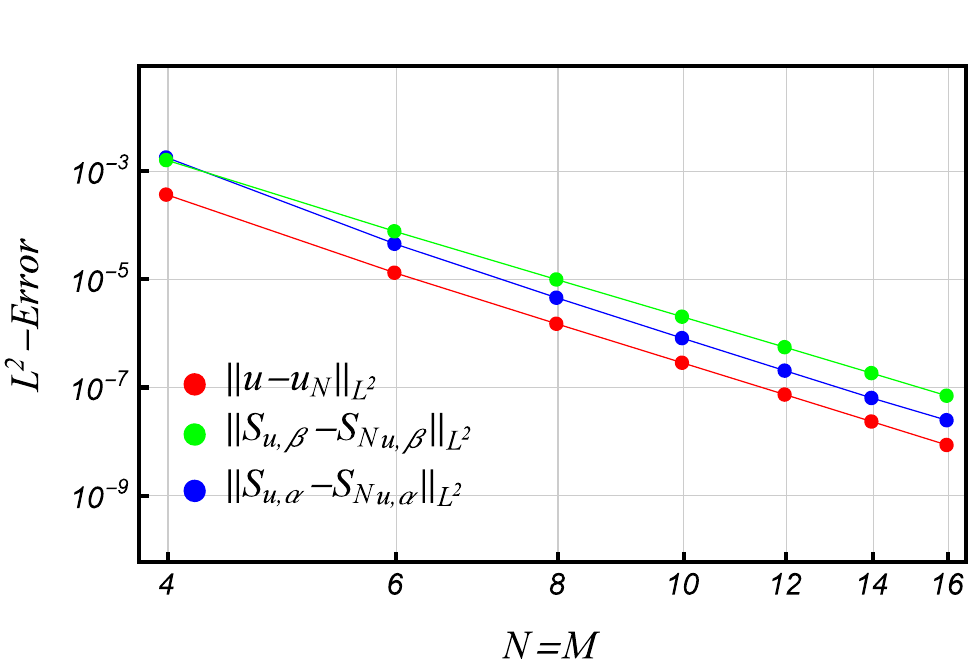}
%		\caption{}
		%\label{fig:}
	\end{subfigure}
	\begin{subfigure}{0.49\textwidth}
		\centering
		\includegraphics[width=1\linewidth]{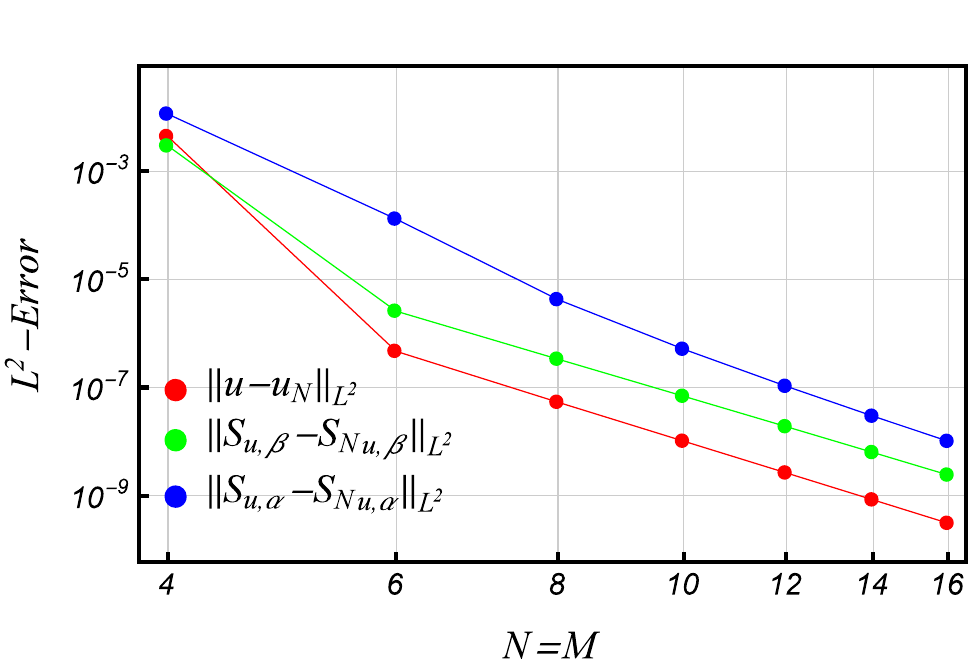}
%		\caption{}
		%\label{fig:}
	\end{subfigure}
	\vspace{-0.1 in}
	\caption{PG spectral method, $L^2$-norm convergence study: $(1+1)$-d FPDE adjoint to corresponding FSEs with one-sided fractional derivative, $k=1$, $\alpha/2 = 0.25$, and $\beta/2 = 0.75$, for Case I (left) and Case II (right), where $N = M$.}
	\label{Fig: uext suext error case I}
\end{figure}
%
%******************************************************************************************

%\newpage

$\bullet$ \textbf{Fractional Model Construction.} The second part of numerical results is dedicated to study the efficiency of developed iterative algorithm in obtaining the set of model parameters $q$ (fractional indices) and thus, construct the fractional model. We test our developed scheme by method of fabricated solution, assuming a given set of input (exact solution) and output (force term) for our fractional model.  

We begin with a fractional IVP of the form $\prescript{}{0}{\mathcal{D}}_{t}^{\alpha} u(t) = f(t)$, $\alpha\in(0,1) $, and assume that the exact solution and force function are given as, 
\begin{align*}
&u^{*}(t)= sin(5 \pi \, \alpha^*/2) \, t^{3 + \alpha^*/2},
\\
&f^{*}(t)= sin(5 \pi \, \alpha^*/2) \, \frac{\Gamma(4+\alpha^*/2)}{\Gamma(4-\alpha^*/2)}  \, t^{3 - \alpha^*/2},
\end{align*}
and the fractional order $\alpha$ is the unknown model parameter. We start from an initial guess $\alpha^0$ and use the developed iterative algorithm to converge to the true value of fractional index $\alpha$. We also consider a fractional BVP of the form $\prescript{}{-1}{\mathcal{D}}_{x}^{\beta} u(x) = f(x)$, $\beta\in(1,2)$, and assume that the exact solution and force function are given as,
\begin{align*}
&u^{*}(x) = (1+x)^{3+\beta^*/2}-\frac{1}{2}(1+x)^{4+\beta^*/2}
\\
&f^{*}(x) = \frac{\Gamma(4+\beta^*/2)}{\Gamma(4-\beta^*/2)} (1+x)^{3+\beta^*/2} - \frac{1}{2} \frac{\Gamma(5+\beta^*/2)}{\Gamma(5-\beta^*/2)}  (1+x)^{4+\beta^*/2}
\end{align*} 
and the fractional order $\beta$ is the unknown model parameter. We again use the developed iterative algorithm to capture the true value of fractional index $\beta$, starting from an initial guess $\beta^0$.

Tables \ref{Table: FIVP model construction} and \ref{Table: FBVP model construction} show two examples for each case of fractional IVP and BVP, where the true values of fractional orders are $\alpha^* = 0.3$, $\alpha^* = 0.9$, $\beta^* = 1.1$, and $\beta^* = 1.7$. We observe that the proposed iterative formulation converges accurately to the exact values with in few numbers of iterations. We note that in the case of fractional IVP and BVP, the search region is already small enough so that the nearby solution is valid, and therefore, we only need to perform the second stage of iterative algorithm.

\begin{table}[t]
	\center
	\caption{\label{Table: FIVP model construction} Fractional model construction for the two cases of fractional IVP.}
	\vspace{-0.1 in}
	\scalebox{1}{
		\begin{tabular}
			%
			%{r@{\hspace*{.30in}}c@{\hspace*{.30in}}c@{\hspace*{.30in}}c@{\hspace*{.30in}}c@{\hspace*{.30in}}c@{\hspace*{.30in}}c} 
			{c@{\hspace*{.40in}} c@{\hspace*{.40in}} c }
			Iteration Index & \multicolumn{2}{c}{$\prescript{}{0}{\mathcal{D}}_{t}^{\alpha} u(t) = f(t)  $}   \\
			\hline
			\hline 
			i & $\alpha^i$ &  $\alpha^i$   \\
			\hline
			initial guess &  $0.300000$ & $0.900000$  \\
			\hline
			1 &  $0.971980$ & $0.320520$   \\
			\hline
			2 &  $0.900418$ & $0.300068$   \\
			\hline
			3 &  $0.900000$ & $0.300000$  \\
			\hline
			\hline
			True Value &  $0.9$ & $0.3$  \\
			\hline
			%\vspace{0.2 in}
			%
		\end{tabular}
	}
\end{table}
\begin{table}[t]
	\center
	\caption{\label{Table: FBVP model construction} Fractional model construction for the two cases of fractional BVP.}
	\vspace{-0.1 in}
	\scalebox{1}{
		\begin{tabular}
			%
			%{r@{\hspace*{.30in}}c@{\hspace*{.30in}}c@{\hspace*{.30in}}c@{\hspace*{.30in}}c@{\hspace*{.30in}}c@{\hspace*{.30in}}c} 
			{c@{\hspace*{.40in}} c@{\hspace*{.40in}} c }
			Iteration Index  & \multicolumn{2}{c}{$\prescript{}{-1}{\mathcal{D}}_{x}^{\beta} u(x) = f(x)  $} \\
			\hline
			\hline 
			i &  $\beta^i$ & $\beta^i$ \\
			\hline
			initial guess & $1.100000$ & $1.9000000$  \\
			\hline
			1 & $1.882020$ & $1.228040$  \\
			\hline
			2 & $1.708120$ & $1.106096$  \\
			\hline
			3 & $1.700020$ & $1.100016$  \\
			\hline
			4 & $1.700000$ & $1.100000$  \\
			\hline
			\hline
			True Value & $1.7$ & $1.1$  \\
			\hline
			%\vspace{0.2 in}
			%
		\end{tabular}
	}
\end{table}
%

%\newpage

Moreover, we consider FPDE of the form $ \prescript{}{0}{\mathcal{D}}_{t}^{\alpha} u - k \prescript{}{-1}{\mathcal{D}}_{x}^{\beta} \, u = f  $. We assume the exact solution $u^{*} = t^{1+\alpha^*/2} \left( (1+x)^{3+\beta^*/2}-\frac{1}{2}(1+x)^{4+\beta^*/2} \right)$ and plug it into the FPDE with given $ \{\alpha^* , \beta^*\} $ to obtain the exact force function $f^*$. We study the example, in which, $ \{\alpha^* , \beta^*\} = \{0.1 , 1.64\}$. We perform the two stages of iterative algorithm, where in the first stage, we shrink down the search region 16 time smaller than the original size, by computing the model error at 8 points (See Fig. \ref{Fig: Iterative solver FDE 1}, right). Then, in the next stage, we start from the initial guess $\{ \alpha^0 ,\beta^0 \}=\{ 0.125,1.75 \}$, and observe that the developed iterative method converges to a close neighborhood of true values $\{0.1 , 1.64\}$ within $10^{-3}$ tolerance.
%
%******************************************************************************************
%
\begin{figure}[t]
	\centering
	\includegraphics[width=0.9\linewidth]{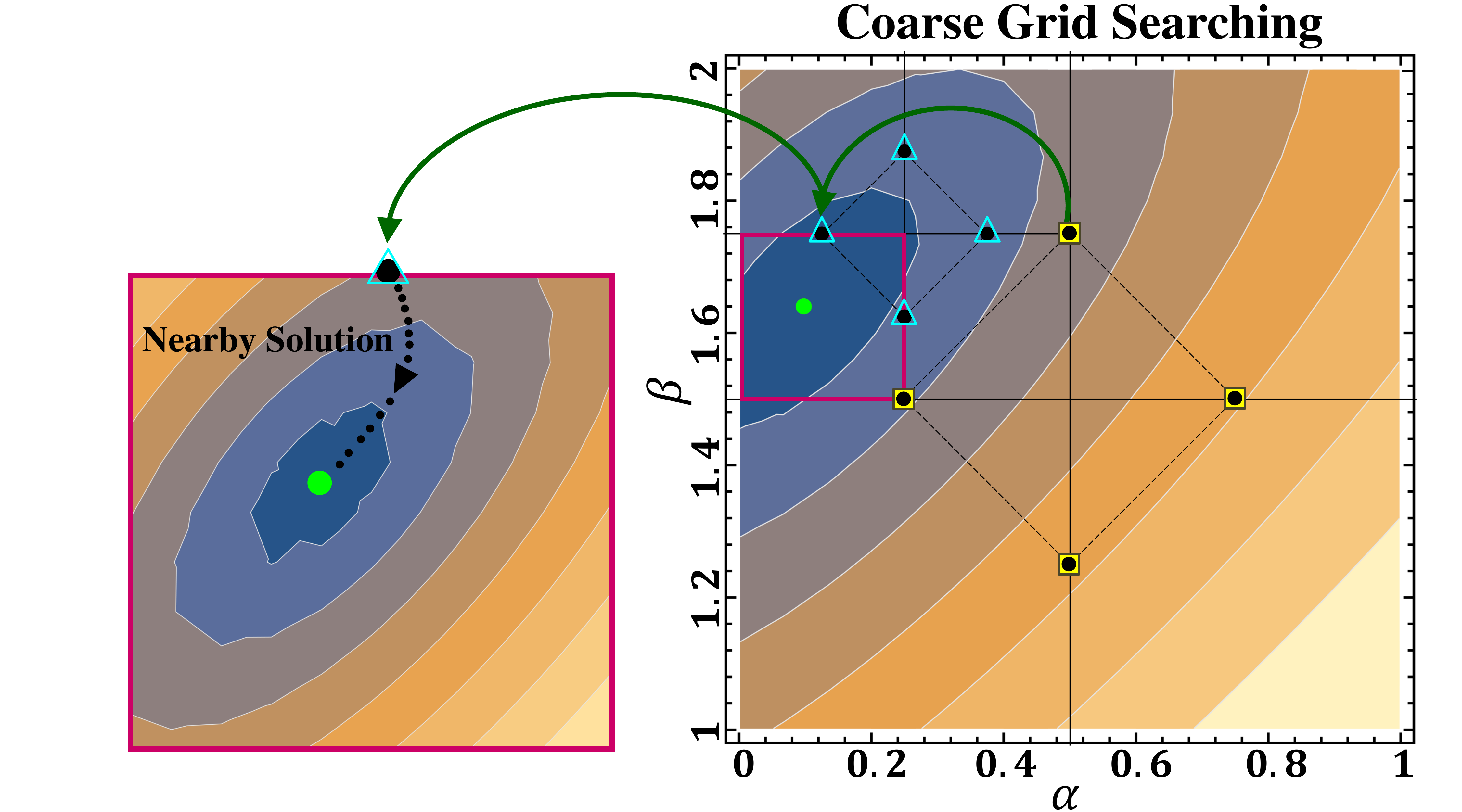}
	\vspace{-0.1 in}	
	\caption{Fractional model construction for the case FPDE, using FSEM based iterative algorithm. The true values of fractional indices are $ \{\alpha^* , \beta^*\} = \{0.1 , 1.64\}$.}
	\label{Fig: Iterative solver FDE 1}
\end{figure}
%
%****************************************************************************************** 

The developed model construction method can also be applied in formulating fractional models to study complex time-varying nonlinear fluid-solid interaction phenomena \cite{atanackovic2014,afzali2016vibrational,afzali2017analysis} and also the effect of damping in structural vibrations \cite{zamani2015asymmetric,abdullatif2018stabilizing}.

%%
%\newpage

%
%%%%%%%%%%%%%%%%%%%%%%%%%%%%%%%%%%%%%%%%%
\section{Summary}
\label{Sec: Summary and Conclusion} 
%%%%%%%%%%%%%%%%%%%%%%%%%%%%%%%%%%%%%%%%%
%
We developed a fractional sensitivity equation method (FSEM) in order to analyze the sensitivity of fractional models (FIVPs, FBVPs, and FPDEs) with respect to their parameters. We derived the adjoint governing dynamics of sensitivity coefficients, i.e. fractional sensitivity equations (FSEs), by taking the partial derivative of FDE with respect to the model parameters, and showed that they preserve the structure of original FDE. We also introduced a new fractional operator, associated with logarithmic-power law kernel, for the first time in the context of FSEM. We extended the existing proper underlying function spaces to respect the extra regularities imposed by FSEs and proved the well-posedness of problem. Moreover, we developed a Petrov-Galerkin (PG) spectral method by employing Jacobi polyfractonomials and Legendre polynomials as basis/test functions, and proved its stability. We further used the developed FSEM to formulate an optimization problem in order to construct the fractional model by estimating the model parameters. We defined two types of model error as objective functions and proposed a two-stages search algorithm to minimize them. We presented the steps of iterative algorithm in a pseudo code. Finally, we examined the performance of proposed numerical scheme in solving coupled FPDE and FSEs, where we numerically study the convergence rate of error. We also investigated the efficiency of developed iterative algorithm in estimating the derivative order for different cases of fractional models.

%%
%%
%%
%%+++++++++++++++++++++++++++++++++++++++
%\section*{Acknowledgements}
%%+++++++++++++++++++++++++++++++++++++++
%
%This work was supported by 

%newpage
\appendix
%
%%%%%%%%%%%%%%%%%%%%%%%%%%%%%
\section{Proof of Lemma \eqref{lem: Caputo vs. Riemann PowLog}}
\label{Sec: App. proof of RL Caputo PL derivative}
%%%%%%%%%%%%%%%%%%%%%%%%%%%%%
%

\noindent Part A: $\sigma \in (0,1)$. We start from the $RL-PL$ definition, given in \eqref{Eq: left sided LP kernel derivative}.
\begin{align}
\prescript{RL-LP}{a}{\mathcal{D}}_{x}^{\sigma} u 
&= 
\frac{1}{\Gamma(1-\sigma)}  \frac{d}{dx} \int_{a}^{x} \, (x - s)^{-\sigma} \, \log(x-s) \, u(s) \, ds, \, \text{ (integrate by parts)}
\\
\nonumber
&=
\frac{1}{\Gamma(1-\sigma)}  \frac{d}{dx}
\Big\lbrace
\frac{u(s) \, (x-s)^{1-\sigma}}{(-\sigma+1)^2} (1-(-\sigma+1) \, \log(x-s))  \Bigg|_{s=a}^{s=x}
\\
\nonumber
& \qquad
- \int_{a}^{x} \, \frac{(x - s)^{-\sigma+1}}{(-\sigma+1)^2}  \, (1-(-\sigma+1) \, \log(x-s)) \, u'(s) \, ds
\Big\rbrace , 
\\
\nonumber
&=
\frac{1}{\Gamma(1-\sigma)}  \frac{d}{dx}
\Big\lbrace
\frac{u(a) \, (x-a)^{1-\sigma}}{(-\sigma+1)^2} (1-(-\sigma+1) \, \log(x-a))
\\
\nonumber
& \qquad
- \int_{a}^{x} \, \frac{(x - s)^{-\sigma+1}}{(-\sigma+1)^2}  \, (1-(-\sigma+1) \, \log(x-s)) \, u'(s) \, ds
\Big\rbrace ,
\\
\nonumber
&=
\frac{u(a)}{\Gamma(1-\sigma)}  \, \frac{\log(x-a)}{(x-a)^{\sigma}} 
+ \frac{1}{\Gamma(1-\sigma)} \int_{a}^{x} \, \frac{\log(x-s)}{(x - s)^{-\sigma}} \, u'(s) \, ds, \, \text{ (by Leibnitz rule)}
\\
\nonumber
&=
\frac{u(a)}{\Gamma(1-\sigma)} \, \frac{\log(x-a)}{(x-a)^{\sigma}} 
+ \prescript{C-LP}{a}{\mathcal{D}}_{x}^{\sigma} u
\end{align}

\noindent Part B: $\sigma \in (1,2)$. Similarly, we start from the $RL-PL$ definition, given in \eqref{Eq: left sided LP kernel derivative}.
\begin{align}
\prescript{RL-LP}{a}{\mathcal{D}}_{x}^{\sigma} u 
&= 
\frac{1}{\Gamma(2-\sigma)}  \frac{d^2}{dx^2} \int_{a}^{x} \, (x - s)^{-\sigma+1} \, \log(x-s) \, u(s) \, ds, \, \text{ (integrate by parts twice)}
\\
\nonumber
&=
\frac{1}{\Gamma(2-\sigma)}  \frac{d^2}{dx^2}
\Big\lbrace
\frac{u(s) \, (x-s)^{-\sigma+2}}{(-\sigma+2)^2} (1-(-\sigma+2)\log(x-s)) \Bigg|_{s=a}^{s=x}
\\
\nonumber
& \qquad
- \frac{u'(s) \, (x-s)^{-\sigma+3}}{(-\sigma+2)^2(-\sigma+3)^2}  \left(1-2(-\sigma+3)+(-\sigma+3)(-\sigma+2)\log(x-s) \right) \Bigg|_{s=a}^{s=x}
\\
\nonumber
& \qquad
+ \int_{a}^{x} \, \frac{(x-s)^{-\sigma+3}}{(-\sigma+2)^2(-\sigma+3)^2}  \, \left(1-2(-\sigma+3)+(-\sigma+3)(-\sigma+2)\log(x-s) \right) \, u''(s) \, ds
\Big\rbrace,
\\
\nonumber
&=
\frac{1}{\Gamma(2-\sigma)}  \frac{d^2}{dx^2}
\Big\lbrace
\frac{u(a) \, (x-a)^{-\sigma+2}}{(-\sigma+2)^2} (1-(-\sigma+2)\log(x-a))
\\
\nonumber
& \qquad
- \frac{u'(a) \, (x-a)^{-\sigma+3}}{(-\sigma+2)^2(-\sigma+3)^2}  \left(1-2(-\sigma+3)+(-\sigma+3)(-\sigma+2)\log(x-a) \right)
\\
\nonumber
& \qquad
+ \int_{a}^{x} \, \frac{(x-s)^{-\sigma+3}}{(-\sigma+2)^2(-\sigma+3)^2}  \, \left(1-2(-\sigma+3)+(-\sigma+3)(-\sigma+2)\log(x-s) \right) \, u''(s) \, ds
\Big\rbrace,
\\
\nonumber
&=
\frac{u(a)}{\Gamma(2-\sigma)} \frac{1+(-\sigma+1)\log(x-a)}{(x-a)^{\sigma}} 
+ \frac{u'(a)}{\Gamma(2-\sigma)} \frac{\log(x-a)}{(x-a)^{\sigma-1}} 
\\
\nonumber
& \qquad 
+ \frac{1}{\Gamma(2-\sigma)} \int_{a}^{x} \, (x-s)^{-\sigma+1} \,  \log(x-s) \, u''(s) \, ds ,  \, \text{ (by Leibnitz rule)}
\\
\nonumber
&=
\frac{u(a)}{\Gamma(1-\sigma)} \frac{1+(-\sigma+1)\log(x-a)}{(x-a)^{\sigma}} + \frac{u'(a)}{\Gamma(1-\sigma)} \frac{\log(x-a)}{(x-a)^{\sigma-1}}  +   \prescript{C-PL}{a}{\mathcal{D}}_{x}^{\sigma} u.
\end{align}
%	

%
%%%%%%%%%%%%%%%%%%%%%%%%%%%%%
\section{Proof of Lemma \eqref{Lem: norm equivalence 1}}
\label{Sec: App. proof of norm equivalence}
%%%%%%%%%%%%%%%%%%%%%%%%%%%%%
%

In Lemma 2.1 in \cite{Li2010} and also in \cite{ervin2007variational}, it is shown that $\Vert \cdot \Vert_{{^l}H^{\sigma}_{}(\Lambda)}$ and $\Vert \cdot \Vert_{{^r}H^{\sigma}_{}(\Lambda)}$ are equivalent. Therefore, for $u \in H^{\sigma}_{}(\Lambda)$, there exist positive constants $C_1$ and $C_2$ such that
\begin{align}
\Vert u \Vert_{{}H^{\sigma}_{}(\Lambda)} \leq C_1 \Vert u \Vert_{{^l}H^{\sigma}_{}(\Lambda)},
\quad
\Vert u \Vert_{{}H^{\sigma}_{}(\Lambda)} \leq C_2 \Vert u \Vert_{{^r}H^{\sigma}_{}(\Lambda)}, 
\end{align}	
which leads to
\begin{align}
\Vert u \Vert_{{}H^{\sigma}_{}(\Lambda)}^2 &\leq C_1^2 \Vert u \Vert_{{^l}H^{\sigma}_{}(\Lambda)}^2 +  C_2^2 \Vert u \Vert_{{^r}H^{\sigma}_{}(\Lambda)}^2 ,
\nonumber
\\
&=  C_1^2 \,\Vert \prescript{}{a}{\mathcal{D}}_{x}^{\sigma}\, (u)\Vert_{L^2(\Lambda)}^2+ C_2^2 \,\Vert \prescript{}{x}{\mathcal{D}}_{b}^{\sigma}\, (u)\Vert_{L^2(\Lambda)}^2+(C_1^2+C_2^2)\, \Vert u \Vert_{L^2(\Lambda)}^2 ,
\nonumber
\\
&\leq \tilde{C}_1 \,\Vert u \Vert_{{^c}H^{\sigma}_{}(\Lambda)}^2,
\end{align}	
where $\tilde{C}_1$ is a positive constant. Similarly, we can show that $\Vert u \Vert_{{^c}H^{\sigma}_{}(\Lambda)}^2 \leq \tilde{C}_2 \, \Vert u \Vert_{{}H^{\sigma}_{}(\Lambda)}$, where $\tilde{C}_2$ is a positive constant.

%
%%%%%%%%%%%%%%%%%%%%%%%%%%%%%
%\section{Proof of Lemma \eqref{norm_221}}
\section{Proof of Lemma \eqref{space norm 1}}
\label{Sec: App. proof of norm Xd}
%%%%%%%%%%%%%%%%%%%%%%%%%%%%%
%
$\mathcal{X}_1$ is endowed with the norm $\Vert \cdot \Vert_{\mathcal{X}_1}$, where $\Vert \cdot \Vert_{\mathcal{X}_1}\equiv \Vert \cdot \Vert_{{^c}H^{\beta_1/2}_{}(\Lambda_1)}$ by Lemma \ref{Lem: norm equivalence 1}. Moreover, $\mathcal{X}_2$ is associated with the norm 
\begin{equation}
\label{eq234}
\Vert \cdot \Vert_{\mathcal{X}_2} \equiv \bigg{\{} \Vert \cdot \Vert_{{^c}H^{\beta_2/2}_0 \Big((a_2,b_2); L^2(\Lambda_{1}) \Big)}^2 + \Vert \cdot \Vert_{ L^2\Big((a_2,b_2); \mathcal{X}_{1}\Big)}^2 \bigg{\}}^{\frac{1}{2}},
\end{equation}	
where
\begin{align}
\Vert u \Vert_{{^c}H^{\beta_2/2}_0 \Big((a_2,b_2); L^2(\Lambda_{1}) \Big)}^2 
&= \int_{a_1}^{b_1}\, \Big( \int_{a_2}^{b_2}\,  \vert \prescript{}{a_2}{\mathcal{D}}_{x_2}^{\beta_2/2} u  \vert^2  \,  dx_2 + \int_{a_2}^{b_2}\,  \vert \prescript{}{x_2}{\mathcal{D}}_{b_2}^{\beta_2/2} u \vert^2  \,  dx_2 + \int_{a_2}^{b_2}\,  \vert  u \vert^2  \,  dx_2 \Big) \,dx_1
\nonumber
\\
&= \int_{a_1}^{b_1}\int_{a_2}^{b_2}\,  \vert \prescript{}{a_2}{\mathcal{D}}_{x_2}^{\beta_2/2} u  \vert^2  \,  dx_2 dx_1 + \int_{a_1}^{b_1}\int_{a_2}^{b_2}\,  \vert \prescript{}{x_2}{\mathcal{D}}_{b_2}^{\beta_2/2} u \vert^2  \,  dx_2 dx_1
+ \int_{a_1}^{b_1} \int_{a_2}^{b_2}\,  \vert  u \vert^2  \,  dx_2 dx_1
\nonumber
\\
&=
\Vert \prescript{}{x_2}{\mathcal{D}}_{b_2}^{\beta_2/2}\, (u)\Vert_{L^2(\Lambda_d)}^2+\Vert \prescript{}{a_2}{\mathcal{D}}_{x_2}^{\beta_2/2}\, (u)\Vert_{L^2(\Lambda_d)}^2+\Vert u \Vert_{L^2(\Lambda_d)}^2,
\end{align}
and
\begin{align}
&\Vert u \Vert_{L^2\Big((a_2,b_2); \mathcal{X}_{1}\Big)}^2
\nonumber
\\
&= \int_{a_2}^{b_2}\, \Big( \int_{a_1}^{b_1}\,  \vert \prescript{}{a_1}{\mathcal{D}}_{x_1}^{\beta_1/2} u  \vert^2  \,  dx_1 + \int_{a_1}^{b_1}\,  \vert \prescript{}{x_1}{\mathcal{D}}_{b_1}^{\beta_1/2} u \vert^2  \,  dx_1 + \int_{a_1}^{b_1}\,  \vert u \vert^2  \,  dx_1 \Big) \,dx_2
\nonumber
\\
&= \int_{a_2}^{b_2}\int_{a_1}^{b_1}  \vert \prescript{}{a_1}{\mathcal{D}}_{x_1}^{\beta_1/2} u  \vert^2    dx_1 dx_2 + \int_{a_2}^{b_2}\int_{a_1}^{b_1}  \vert \prescript{}{x_1}{\mathcal{D}}_{b_1}^{\beta_1/2} u \vert^2    dx_1 dx_2
+ \int_{a_2}^{b_2}\int_{a_1}^{b_1}  \vert  u \vert^2    dx_1 dx_2
\nonumber
\\
&=
\Vert \prescript{}{x_1}{\mathcal{D}}_{b_1}^{\beta_1/2}\, u\Vert_{L^2(\Lambda_2)}^2+\Vert \prescript{}{a_1}{\mathcal{D}}_{x_1}^{\beta_1/2}\, u\Vert_{L^2(\Lambda_2)}^2+\Vert u \Vert_{L^2(\Lambda_2)}^2.
\end{align}
We use the mathematical induction to carry out the proof. Therefore, we assume the following equality holds
\begin{equation}
\label{eq23421}
\Vert \cdot \Vert_{\mathcal{X}_{k-1}} \equiv \bigg{\{} \sum_{i=1}^{k-1} \Big(\Vert \prescript{}{x_i}{\mathcal{D}}_{b_i}^{\beta_i/2}\, (\cdot)\Vert_{L^2(\Lambda_{k-1})}^2+\Vert \prescript{}{a_i}{\mathcal{D}}_{x_i}^{\beta_i/2}\, (\cdot)\Vert_{L^2(\Lambda_{k-1})}^2 \Big) + \Vert  \cdot \Vert_{L^2(\Lambda_{k-1})}^2 \bigg{\}}^{\frac{1}{2}}.
\end{equation}
Since,
\begin{align*}
&\Vert u \Vert_{{^c}H^{\beta_k/2}_0 \Big((a_k,b_k); L^2(\Lambda_{k-1}) \Big)}^2 
\nonumber
\\
&= \int_{\Lambda_{k-1}}^{}\, \Big( \int_{a_k}^{b_k}\,  \vert \prescript{}{a_k}{\mathcal{D}}_{x_k}^{\beta_k/2} u  \vert^2  \,  dx_k + \int_{a_k}^{b_k}\,  \vert \prescript{}{x_k}{\mathcal{D}}_{b_k}^{\beta_k/2} u \vert^2  \,  dx_k + \int_{a_k}^{b_k}\,  \vert u \vert^2  \,  dx_k \Big) \,d\Lambda_{k-1}
\nonumber
\\
&= \int_{\Lambda_{k-1}}^{}\int_{a_k}^{b_k}\,  \vert \prescript{}{a_k}{\mathcal{D}}_{x_k}^{\beta_k/2} u  \vert^2  \,  dx_k d\Lambda_{k-1} + \int_{\Lambda_{k-1}}^{}\int_{a_k}^{b_k}\,  \vert \prescript{}{x_k}{\mathcal{D}}_{b_k}^{\beta_k/2} u \vert^2  \,  dx_k d\Lambda_{k-1} + \int_{\Lambda_{k-1}}^{}\int_{a_k}^{b_k}\,  \vert u \vert^2  \,  dx_k d\Lambda_{k-1}
\nonumber
\\
&=
\Vert \prescript{}{x_k}{\mathcal{D}}_{b_k}^{\beta_k/2}\, (u)\Vert_{L^2(\Lambda_k)}^2+\Vert \prescript{}{a_k}{\mathcal{D}}_{x_k}^{\beta_k/2}\, (u)\Vert_{L^2(\Lambda_k)}^2+\Vert u \Vert_{L^2(\Lambda_k)}^2,
\end{align*}
and
\begin{align*}
\Vert u \Vert_{L^2\Big((a_k,b_k); \mathcal{X}_{k-1}\Big)}^2 
&= \int_{a_k}^{b_k} 
\left(  
\sum_{i=1}^{k-1} \left(
\int_{\Lambda_{k-1}}^{}  \vert \prescript{}{a_i}{\mathcal{D}}_{x_i}^{\beta_i/2} u  \vert^2   d\Lambda_{k-1} + \int_{\Lambda_{k-1}}^{}  \vert \prescript{}{x_i}{\mathcal{D}}_{b_i}^{\beta_i/2} u \vert^2   d\Lambda_{k-1} \right)
+ \int_{\Lambda_{k-1}}^{}  \vert u \vert^2    d\Lambda_{k-1} 
\right) dx_k
\nonumber
\\
&= \sum_{i=1}^{k-1} \Big( \int_{\Lambda_k}^{} \vert \prescript{}{a_i}{\mathcal{D}}_{x_i}^{\beta_i/2} u  \vert^2    d\Lambda_k +  \int_{\Lambda_k}^{}  \vert \prescript{}{x_i}{\mathcal{D}}_{b_i}^{\beta_i/2} u \vert^2    d\Lambda_k \Big)
+ \int_{\Lambda_k}^{}  \vert  u \vert^2     d\Lambda_k
\nonumber
\\
&=
\sum_{i=1}^{k-1} \Big( \Vert \prescript{}{x_i}{\mathcal{D}}_{b_i}^{\beta_i/2}\, u\Vert_{L^2(\Lambda_k)}^2+\Vert \prescript{}{a_i}{\mathcal{D}}_{x_i}^{\beta_i/2}\, u\Vert_{L^2(\Lambda_k)}^2 \Big)+\Vert u \Vert_{L^2(\Lambda_k)}^2,
\end{align*}
we can show that 
\begin{equation}
\Vert \cdot \Vert_{\mathcal{X}_{k}} \equiv \bigg{\{} \sum_{i=1}^{k} \Big(\Vert \prescript{}{x_i}{\mathcal{D}}_{b_i}^{\beta_i/2}\, (\cdot)\Vert_{L^2(\Lambda_{k})}^2+\Vert \prescript{}{a_i}{\mathcal{D}}_{x_i}^{\beta_i/2}\, (\cdot)\Vert_{L^2(\Lambda_{k})}^2 \Big) + \Vert  \cdot \Vert_{L^2(\Lambda_{k})}^2 \bigg{\}}^{\frac{1}{2}}.
\end{equation}
%

%
%%%%%%%%%%%%%%%%%%%%%%%%%%%%%
\section{Proof of Lemma \eqref{lem_generalize}}
\label{Sec: App. proof of gen int by part}
%%%%%%%%%%%%%%%%%%%%%%%%%%%%%
%
According to \cite{kharazmi2017petrov}, we have $\prescript{}{a_i}{\mathcal{D}}_{x_i}^{\beta_i} u=\prescript{}{a_i}{\mathcal{D}}_{x_i}^{\beta_i/2} (\prescript{}{a_i}{\mathcal{D}}_{x_i}^{\beta_i/2} u)$ and $\prescript{}{x_i}{\mathcal{D}}_{b_i}^{\beta_i/2} u=\prescript{}{x_i}{\mathcal{D}}_{b_i}^{\beta_i/2}(\prescript{}{x_i}{\mathcal{D}}_{b_i}^{\beta_i/2} u)$. Let $\bar{u}=\prescript{}{a_i}{\mathcal{D}}_{x_i}^{\beta_i/2} u$. Then,
\begin{eqnarray}
(\prescript{}{a_i}{\mathcal{D}}_{x_i}^{\beta_i} u,v)_{\Lambda_d}&=& (\prescript{}{a_i}{\mathcal{D}}_{x_i}^{\beta_i/2} \bar{u},v)_{\Lambda_d}=\int_{\Lambda_d}^{} \frac{1}{\Gamma(1-\beta_i/2)}  \big[ \frac{d}{dx_i}\, \int_{a_i}^{x_i}\frac{\bar{u}(s)\,ds}{(x_i-s)^\beta_i/2}  \big]v\,d\Lambda_d
\nonumber
\\
&=& \Big{\{} \frac{v}{\Gamma(1-\beta_i/2)\int_{a_i}^{x_i} \frac{\bar{u}ds}{(x_i-s)^{\beta_i/2}}} \Big{\}}^{b_i}_{x_i=a_i} - \int_{\Lambda_d}^{} \frac{1}{\Gamma(1-\beta_i/2)}\int_{a_i}^{x_i}\frac{\bar{u}(s)\,ds}{(x_i-s)^{\beta_i/2}} \frac{dv}{dx_i}\, d\Lambda_d. \quad \quad
\end{eqnarray}
Based on the homogeneous boundary conditions, $\Big{\{} \frac{v}{\Gamma(1-\beta_i/2)\int_{a_i}^{x_i} \frac{\bar{u}ds}{(x_i-s)^{\beta_i/2}}} \Big{\}}^{b_i}_{x_i=a_i}=0.$ Therefore, 
\begin{eqnarray}
(\prescript{}{a_i}{\mathcal{D}}_{x_i}^{\beta_i} u,v)_{\Lambda_d}&=&-\int_{\Lambda_i}^{}\frac{1}{\Gamma(1-\beta_i/2)}\int_{a_i}^{x_i}\frac{\bar{u}(s)\,ds}{(x_i-s)^{\beta_i/2}} \frac{dv}{dx_i}\, d\Lambda_i.  
\end{eqnarray}
Moreover, we find that 
\begin{eqnarray}
\frac{d}{ds}\, \int_{a_i}^{b_i}\frac{u}{(x_i-s)^{\beta_i/2}}dx_i &=& \frac{d}{ds} \Big{\{} \{\frac{v\, (x_i-s)^{1-\beta_i/2}}{1-\beta_i/2}\}_{x_i=s_i}^{b_i}-\frac{1}{1-\beta_i/2}\int_{s}^{b_i}\frac{dv}{dx_i}(x_i-s)^{1-\beta_i/2}dx_i \Big{\}}
\nonumber
\\
&=&
-\frac{1}{1-\beta_i/2} \int_{s}^{b_i} \frac{dv}{dx_i}(x_i-s)^{1-\beta_i/2}\, dx_i = \int_{s}^{b_i} \frac{\frac{dv}{dx_i}}{(x_i-s)^{\beta_i/2}}\, dx_i.
\end{eqnarray}
Therefore, we get
\begin{eqnarray}
(\prescript{}{a_i}{\mathcal{D}}_{x_i}^{\beta_i/2} \bar{u},v)_{\Lambda_d}=-\int_{\Lambda_d}^{} \frac{1}{\Gamma(1-\nu)_i}\, \bar{u}(s) \big(-\frac{d}{ds} \int_{s}^{b_i}\frac{v}{(x_i-s)^{\beta_i/2}}dx_i\big)\,ds=( \bar{u},\prescript{}{x_i}{\mathcal{D}}_{b_i}^{\beta_i/2}v)_{\Lambda_d}. \quad
\nonumber
\end{eqnarray}

%%%%%%%%%%%%%%%%%%%%%%%%%%%%%%%%%%%%%%%%%%%%%%%%%%%%%%%%%%%%%%%%%%%%%%%%%%%%%%%%%%%%%%%%%%%%%%%%%%%%%%%%%%%%%%%%%%%%%%%%%%%%%%%%%%%%%%%%%%%%%

%
%%%%%%%%%%%%%%%%%%%%%%%%%%%%%
\section{Proof of Lemma \eqref{norm_223}}
\label{Sec: App. proof of norm time}
%%%%%%%%%%%%%%%%%%%%%%%%%%%%%
%
We know that 
\begin{align*}
\big\vert \left( \prescript{}{0}{\mathcal{D}}_{t}^{\alpha/2} u, \prescript{}{t}{\mathcal{D}}_{T}^{\alpha/2} v \right)_{\Omega} \big\vert
=
\Big( \int_{\Lambda_d}^{} \int_{0}^{T}  \vert \prescript{}{0}{\mathcal{D}}_{t}^{\alpha/2} u \, \prescript{}{t}{\mathcal{D}}_{T}^{\alpha/2} v \vert^2\, dt d\Lambda_d \Big)^{\frac{1}{2}}. 
\end{align*}
Therefore, by H\"{o}lder inequality
\begin{align*}
&
\big\vert \left( \prescript{}{0}{\mathcal{D}}_{t}^{\alpha/2} u, \prescript{}{t}{\mathcal{D}}_{T}^{\alpha/2} v \right)_{\Omega} \big\vert
\\
&\leq \Big( \int_{\Lambda_d}^{} \int_{0}^{T}  \vert \prescript{}{0}{\mathcal{D}}_{t}^{\alpha/2} u \vert^2\, dt d\Lambda_d \Big)^{\frac{1}{2}} \, \Big( \int_{\Lambda_d}^{} \int_{0}^{T}  \vert \prescript{}{t}{\mathcal{D}}_{T}^{\alpha/2} v \vert^2\, dt d\Lambda_d \Big)^{\frac{1}{2}} 
\\
&\leq
\Big( \int_{\Lambda_d}^{} \int_{0}^{T}  \vert \prescript{}{0}{\mathcal{D}}_{t}^{\alpha/2} u \vert^2\, dt d\Lambda_d + \int_{\Lambda_d}^{} \int_{0}^{T}  \vert u \vert^2\, dt d\Lambda_d \Big)^{\frac{1}{2}} \, \Big( \int_{\Lambda_d}^{} \int_{0}^{T}  \vert \prescript{}{t}{\mathcal{D}}_{T}^{\alpha/2} v \vert^2\, dt d\Lambda_d + \int_{\Lambda_d}^{} \int_{0}^{T}  \vert  v \vert^2\, dt d\Lambda_d \Big)^{\frac{1}{2}} 
\\
&= \Vert \prescript{}{0}{\mathcal{D}}_{t}^{\alpha/2} u \Vert_{L^2(\Omega)} \, \Vert \prescript{}{t}{\mathcal{D}}_{T}^{\alpha/2} v \Vert_{L^2(\Omega)} = \Vert u \Vert_{\prescript{l}{}H^{\alpha/2}(I; L^2(\Lambda_d))} \, \Vert v \Vert_{\prescript{r}{}H^{\alpha/2}(I; L^2(\Lambda_d))}.
\end{align*}
Moreover, by equivalence of $\vert \cdot \vert_{H^{s}(I)} \equiv \vert \cdot \vert^{*}_{H^{s}(I)} = \vert \cdot \vert^{1/2}_{{^l}H^{s}(I)} \vert \cdot \vert^{1/2}_{{^r}H^{s}(I)} $ we have
\begin{eqnarray}
\vert( \prescript{}{0}{\mathcal{D}}_{t}^{\alpha/2} u, \prescript{}{t}{\mathcal{D}}_{T}^{\alpha/2} v)_{I}\vert&=&\int_{0}^{T}  \vert \prescript{}{0}{\mathcal{D}}_{t}^{\alpha/2} u \, \prescript{}{t}{\mathcal{D}}_{T}^{\alpha/2} v \vert^2\, dt
\nonumber
\\
&\geq& 
\int_{0}^{T}  \vert \prescript{}{0}{\mathcal{D}}_{t}^{\alpha/2} u \vert^2 dt \, \int_{0}^{T} \vert \prescript{}{t}{\mathcal{D}}_{T}^{\alpha/2} v \vert^2\, dt
\geq \tilde{\beta}_1 \Vert u \Vert_{{^l}H^{s}(I)} \Vert v \Vert_{{^r}H^{s}(I)},
\end{eqnarray}  
where $0<\tilde{\beta}_1\leq 1$. Therefore,
\begin{eqnarray}
\vert( \prescript{}{0}{\mathcal{D}}_{t}^{\alpha/2} u, \prescript{}{t}{\mathcal{D}}_{T}^{\alpha/2} v)_{\Omega} \vert^2 
&=&
\int_{\Lambda_d}^{} \int_{0}^{T}  \vert \prescript{}{0}{\mathcal{D}}_{t}^{\alpha/2} u \, \prescript{}{t}{\mathcal{D}}_{T}^{\alpha/2} v \vert^2\, dt \, d\Lambda_d  
\nonumber
\\
&\geq& 
\int_{\Lambda_d}^{}  \Big{(}\int_{0}^{T}  \vert \prescript{}{0}{\mathcal{D}}_{t}^{\alpha/2} u \vert^2 dt \, \int_{0}^{T} \vert \prescript{}{t}{\mathcal{D}}_{T}^{\alpha/2} v \vert^2\, dt\Big{)} \, d\Lambda_d
\nonumber
\\
&\geq& \bar{\beta} \int_{\Lambda_d}^{}  \int_{0}^{T}  \vert \prescript{}{0}{\mathcal{D}}_{t}^{\alpha/2} u \vert^2 dt d\Lambda_d\, \int_{\Lambda_d}^{} \int_{0}^{T} \vert \prescript{}{t}{\mathcal{D}}_{T}^{\alpha/2} v \vert^2\, dt \, \Lambda_d
\nonumber
\\
&\geq&
\bar{\beta} \tilde{\beta}_2 \Vert u \Vert_{{^l}H^{s}(I)} \Vert v \Vert_{{^r}H^{s}(I)},
\end{eqnarray}
where $0<\tilde{\beta}_2\leq 1$ and $0<\bar{\beta}$.

%
%%%%%%%%%%%%%%%%%%%%%%%%%%%%%
\section{Proof of The Stability Theorem \eqref{inf_sup_d_lem}}
\label{Sec: App. proof of stability}
%%%%%%%%%%%%%%%%%%%%%%%%%%%%%
%

\noindent \textit{Part A}: $d=1$. It is evident that $u$ and $v$ are in Hilbert spaces (see \cite{ervin2007variational,Li2010}). For $0< \tilde{\beta} \leq 1$, we have
\begin{eqnarray}
&&\vert a(u,v)\vert 
\nonumber
\\
&&= \vert (\prescript{}{0}{\mathcal{D}}_{t}^{\alpha/2}\, (u),\prescript{}{t}{\mathcal{D}}_{T}^{\alpha/2}\, (v))_{\Omega} + (\prescript{}{a_1}{\mathcal{D}}_{x_1}^{\beta_1/2}\, (u),\prescript{}{x_1}{\mathcal{D}}_{b_1}^{\beta_1/2}\, (v))_{\Omega}+ (\prescript{}{a_1}{\mathcal{D}}_{x_1}^{\beta_1/2}\, (u),\prescript{}{x_1}{\mathcal{D}}_{b_1}^{\beta_1/2}\, (v))_{\Omega}+(u,v)_{\Omega}\vert
\nonumber
\\
&&\geq \tilde{\beta} \Big(\vert (\prescript{}{0}{\mathcal{D}}_{t}^{\alpha/2}\, (u),\prescript{}{t}{\mathcal{D}}_{T}^{\alpha/2}\, (v))_{\Omega}\vert + \vert (\prescript{}{a_1}{\mathcal{D}}_{x_1}^{\beta_1/2}\, (u),\prescript{}{x_1}{\mathcal{D}}_{b_1}^{\beta_1/2}\, (v))_{\Omega}\vert+\vert (\prescript{}{a_1}{\mathcal{D}}_{x_1}^{\beta_1/2}\, (u),\prescript{}{x_1}{\mathcal{D}}_{b_1}^{\beta_1/2}\, (v))_{\Omega}\vert+\vert(u,v)_{\Omega}\vert\Big),
\nonumber
\end{eqnarray}
since $\underset{u \in U}{\sup} \vert a(u , v)\vert>0$. Next, by equivalence of spaces and their associated norms, \eqref{equiv_space}, and \eqref{equiv_space2}, we obtain
\begin{eqnarray}
\vert (\prescript{}{0}{\mathcal{D}}_{t}^{\alpha/2}\, (u),\prescript{}{t}{\mathcal{D}}_{T}^{\alpha/2}\, (v))_{\Omega}\vert &\geq& C_1 \Vert \prescript{}{0}{\mathcal{D}}_{t}^{\alpha/2} u\Vert_{L^2(\Omega)}
\, \Vert\prescript{}{t}{\mathcal{D}}_{T}^{\alpha/2} v\Vert_{L^2(\Omega)},
\nonumber
\\
\vert (\prescript{}{a_1}{\mathcal{D}}_{x_1}^{\beta_1/2}\, (u),\prescript{}{x_1}{\mathcal{D}}_{b_1}^{\beta_1/2}\, (v))_{\Omega} \vert &\geq& C_2 \Vert \prescript{}{a_1}{\mathcal{D}}_{x_1}^{\beta_1/2} u \Vert_{L^2(\Omega)}\, \Vert \prescript{}{x_1}{\mathcal{D}}_{b_1}^{\beta_1/2} v \Vert_{L^2(\Omega)}, \quad
\nonumber
\end{eqnarray}
and
\begin{eqnarray}
\vert (\prescript{}{x_1}{\mathcal{D}}_{b_1}^{\beta_1/2}\, (u),\prescript{}{a_1}{\mathcal{D}}_{x_1}^{\beta_1/2}\, (v))_{\Omega}\vert \geq C_3 \Vert \prescript{}{x_1}{\mathcal{D}}_{b_1}^{\beta_1/2} u \Vert_{L^2(\Omega)} \, \Vert \prescript{}{a_1}{\mathcal{D}}_{x_1}^{\beta_1/2} v \Vert_{L^2(\Omega)},
\end{eqnarray}
where $C_1$, $C_2$, and $C_3$ are positive constants. Therefore,
\begin{eqnarray}
\label{inequality_eq}
\vert a(u,v)\vert &\geq&
\tilde{C} \tilde{\beta} \Big{\{}  \Vert \prescript{}{0}{\mathcal{D}}_{t}^{\alpha/2} u\Vert_{L^2(\Omega)}
\, \Vert\prescript{}{t}{\mathcal{D}}_{T}^{\alpha/2} v\Vert_{L^2(\Omega)} + \Vert \prescript{}{a_1}{\mathcal{D}}_{x_1}^{\beta_1/2} u \Vert_{L^2(\Omega)}\, \Vert \prescript{}{x_1}{\mathcal{D}}_{b_1}^{\beta_1/2} v \Vert_{L^2(\Omega)} 
\nonumber
\\
&& \quad + \Vert \prescript{}{a_1}{\mathcal{D}}_{x_1}^{\beta_1/2} u \Vert_{L^2(\Omega)} \, \Vert \prescript{}{x_1}{\mathcal{D}}_{b_1}^{\beta_1/2} v \Vert_{L^2(\Omega)} \Big{\}},
\end{eqnarray}
where $\tilde{C}$ is $min\{C_1, \, C_2, \, C_3 \}$. Also, the norm $ \Vert u \Vert_{U} \, \Vert v \Vert_{V}$ is equivalent to the right hand side of inequality \eqref{inequality_eq}. Therefore, $\vert a(u,v)\vert \geq C \, \Vert u \Vert_{U}\Vert v \Vert_{V}$.

\vspace{0.2 in}
\noindent \textit{Part B}: $d > 1$. Similarly, we have 
\begin{align}
\label{bilinear_ineq}
\vert a(u,v)\vert \geq \beta \bigg(\vert (\prescript{}{0}{\mathcal{D}}_{t}^{\alpha/2} (u),\prescript{}{t}{\mathcal{D}}_{T}^{\alpha/2} (v))_{\Omega}\vert + \sum_{i=1}^{d} \Big(\vert (\prescript{}{a_i}{\mathcal{D}}_{x_i}^{\beta_i/2} (u),\prescript{}{x_i}{\mathcal{D}}_{b_i}^{\beta_i/2} (v))_{\Omega}\vert+\vert (\prescript{}{a_i}{\mathcal{D}}_{x_i}^{\beta_i/2} (u),\prescript{}{x_i}{\mathcal{D}}_{b_i}^{\beta_i/2} (v))_{\Omega}\vert\Big) \bigg),
\end{align}
where $0< \beta \leq 1$. Recalling that as the direct consequences of \eqref{equiv_space}, we obtain
\begin{align}
\vert (\prescript{}{a_i}{\mathcal{D}}_{x_i}^{\beta_i/2}\, (u),\prescript{}{x_i}{\mathcal{D}}_{b_i}^{\beta_i/2}\, (v))_{\Omega} \vert 
&\equiv
\Vert \prescript{}{a_i}{\mathcal{D}}_{x_i}^{\beta_i/2}\, (u) \Vert_{L^2(\Omega)} \, \Vert \prescript{}{x_i}{\mathcal{D}}_{b_i}^{\beta_i/2}\, (v)\Vert_{L^2(\Omega)}, \quad 
\nonumber
\\
\vert (\prescript{}{x_i}{\mathcal{D}}_{b_i}^{\beta_i/2}\, (u),\prescript{}{a_i}{\mathcal{D}}_{x_i}^{\beta_i/2}\, (v))_{\Omega} \vert 
&\equiv 
\Vert \prescript{}{x_i}{\mathcal{D}}_{b_i}^{\beta_i/2}\, (u) \Vert_{L^2(\Omega)} \, \Vert \prescript{}{a_i}{\mathcal{D}}_{x_i}^{\beta_i/2}\, (v)\Vert_{L^2(\Omega)}.
\nonumber
\end{align}
Thus,
\begin{align}
\label{qqqq}
&\sum_{i=1}^{d} \Big(\vert (\prescript{}{a_i}{\mathcal{D}}_{x_i}^{\beta_i/2}\, (u),\prescript{}{x_i}{\mathcal{D}}_{b_i}^{\beta_i/2}\, (v))_{\Omega} \vert 
+\vert (\prescript{}{x_i}{\mathcal{D}}_{b_i}^{\beta_i/2}\, (u),\prescript{}{a_i}{\mathcal{D}}_{x_i}^{\beta_i/2}\, (v))_{\Omega} \vert \Big),
\\
\nonumber
&\geq \tilde{C}
\sum_{i=1}^{d} \Big(\Vert \prescript{}{a_i}{\mathcal{D}}_{x_i}^{\beta_i/2}\, (u) \Vert_{L^2(\Omega)} \, \Vert \prescript{}{x_i}{\mathcal{D}}_{b_i}^{\beta_i/2}\, (v)\Vert_{L^2(\Omega)} + \Vert \prescript{}{x_i}{\mathcal{D}}_{b_i}^{\beta_i/2}\, (u) \Vert_{L^2(\Omega)} \, \Vert \prescript{}{a_i}{\mathcal{D}}_{x_i}^{\beta_i/2}\, (v)\Vert_{L^2(\Omega)}\Big),
\\
\nonumber
&\geq \tilde{C}_1 \, \tilde{\beta} \sum_{i=1}^{d} \Big(\Vert \prescript{}{a_i}{\mathcal{D}}_{x_i}^{\beta_i/2}\, (u) \Vert_{L^2(\Omega)}  + \Vert \prescript{}{x_i}{\mathcal{D}}_{b_i}^{\beta_i/2}\, (u) \Vert_{L^2(\Omega)} \Big) 
\times \sum_{j=1}^{d} \Big( \Vert \prescript{}{x_j}{\mathcal{D}}_{b_j}^{\nu_j}\, (v)\Vert_{L^2(\Omega)}, + \Vert \prescript{}{a_j}{\mathcal{D}}_{x_j}^{\nu_j}\, (v)\Vert_{L^2(\Omega)}\Big),
\end{align}
for $u,\, v \in  L^2(I; \mathcal{X}_d)$, where $0<\tilde{C}$ and $0<\tilde{\beta}\leq 1$. Furthermore, Lemma \ref{norm_223} yields 
\begin{equation}
\label{qqqq2}
\vert (\prescript{}{0}{\mathcal{D}}_{t}^{\alpha/2}(u),\prescript{}{t}{\mathcal{D}}_{T}^{\alpha/2}(v))_{\Omega} \vert \equiv  \Vert u \Vert_{\prescript{r}{}H^{\alpha/2}(I; L^2(\Lambda_d))} \, \, \Vert v \Vert_{\prescript{l}{}H^{\alpha/2}(I;L^2(\Lambda_d))}.
\end{equation}
Therefore, from \eqref{qqqq} and \eqref{qqqq2} we have
\begin{eqnarray}
\label{thm123}
\vert a(u,v)\vert 
\geq \beta \Big(
\Vert u \Vert_{\prescript{r}{}H^{\alpha/2}(I; L^2(\Lambda_d))} \, \, \Vert v \Vert_{\prescript{l}{}H^{\alpha/2}(I;L^2(\Lambda_d))} + \Vert u \Vert_{L^2(I; \mathcal{X}_d)} \, \Vert v \Vert_{L^2(I; \mathcal{X}_d)}\Big),
\end{eqnarray}
where
\begin{eqnarray}
\label{thm124}
&&\Vert u \Vert_{\prescript{r}{}H^{\alpha/2}(I; L^2(\Lambda_d))} \, \, \Vert v \Vert_{\prescript{l}{}H^{\alpha/2}(I;L^2(\Lambda_d))} + \Vert u \Vert_{L^2(I; \mathcal{X}_d)} \, \Vert v \Vert_{L^2(I; \mathcal{X}_d)}
\nonumber
\\
&&\geq \tilde{C}_2 \Big(\Vert u \Vert_{\prescript{r}{}H^{\alpha/2}(I; L^2(\Lambda_d))} + \Vert u \Vert_{L^2(I; \mathcal{X}_d)} \Big) \Big(\Vert v \Vert_{\prescript{l}{}H^{\alpha/2}(I;L^2(\Lambda_d))}+\Vert v \Vert_{L^2(I; \mathcal{X}_d)}\Big)
\end{eqnarray}
for $u \in U $, $v \in U$ and $0<\tilde{C}_2\leq 1$. By considering \eqref{thm123} and \eqref{thm124}, we get
\begin{equation}
\label{inequality_eq3}
\vert a(u,v)\vert \geq C \, \Vert u \Vert_{U}\Vert v \Vert_{V}.
\end{equation}

%%\clearpage

\newpage
\bibliographystyle{siam}
\bibliography{RFSLP_Refs3}

\begin{thebibliography}{10}

\bibitem{abdullatif2018stabilizing}
{\sc M.~Abdullatif, R.~Mukherjee, and A.~Hellum}, {\em Stabilizing and
  destabilizing effects of damping in non-conservative systems: Some new
  results}, Journal of Sound and Vibration, 413 (2018), pp.~442--455.

\bibitem{afzali2017analysis}
{\sc F.~Afzali, G.~D. Acar, and B.~F. Feeny}, {\em Analysis of the periodic
  damping coefficient equation based on floquet theory}, in ASME 2017
  International Design Engineering Technical Conferences and Computers and
  Information in Engineering Conference, American Society of Mechanical
  Engineers, 2017, pp.~V008T12A050--V008T12A050.

\bibitem{afzali2016vibrational}
{\sc F.~Afzali, O.~Kapucu, and B.~F. Feeny}, {\em Vibrational analysis of
  vertical-axis wind-turbine blades}, in ASME 2016 International Design
  Engineering Technical Conferences and Computers and Information in
  Engineering Conference. American Society of Mechanical Engineers, 2016.

\bibitem{anastasio1994fractional}
{\sc T.~J. Anastasio}, {\em The fractional-order dynamics of brainstem
  vestibulo-oculomotor neurons}, Biological cybernetics, 72 (1994), pp.~69--79.

\bibitem{atanackovic2014}
{\sc T.~M. Atanackovic, S.~Pilipovic, B.~Stankovic, and D.~Zorica}, {\em
  Fractional calculus with applications in mechanics: vibrations and diffusion
  processes}, John Wiley \& Sons, 2014.

\bibitem{meer01}
{\sc B.~Baeumer, D.~A. Benson, M.M. Meerschaert, and S.~W. Wheatcraft}, {\em
  Subordinated advection-dispersion equation for contaminant transport}, Water
  Resources Research, 37 (2001), pp.~1543--1550.

\bibitem{bhrawy2015spectral}
{\sc A.~H. Bhrawy, E.~H. Doha, D.~Baleanu, and S.~S. Ezz-Eldien}, {\em A
  spectral tau algorithm based on jacobi operational matrix for numerical
  solution of time fractional diffusion-wave equations}, Journal of
  Computational Physics, 293 (2015), pp.~142--156.

\bibitem{Bischof96}
{\sc C.~Bischof, P.~Khademi, A.~Mauer-Oats, and A.~Carle}, {\em Adifor 2.0:
  Automatic differentiation of fortran 77 program}, IEEE Computaitonal Science
  and Engineering,  (1996).

\bibitem{Bischof03}
{\sc C.~Bischof, B.~Land, and A.~Vehreschild}, {\em Automatic differentiation
  for matlab program}, Proceeding in applied mathematics and mechanics, 2003,
  pp.~2:50--53.

\bibitem{Bischof97}
{\sc C.~Bischof, L.~Roh, and A.~Mauer-Oats}, {\em Adic: an extensible automatic
  differentiation tool for ansi-c}, Software Practice and Experience,  (1997).

\bibitem{Cao2013}
{\sc J.~Cao and C.~Xu}, {\em A high order schema for the numerical solution of
  the fractional ordinary differential equations}, Journal of Computational
  Physics, 238 (2013), pp.~154--168.

\bibitem{chakraborty2009parameter}
{\sc P.~Chakraborty, M.~M. Meerschaert, and C.~Y. Lim}, {\em Parameter
  estimation for fractional transport: A particle-tracking approach}, Water
  resources research, 45 (2009).

\bibitem{chen2016fast}
{\sc S.~Chen, F.~Liu, X.~Jiang, I.~Turner, and K.~Burrage}, {\em Fast finite
  difference approximation for identifying parameters in a two-dimensional
  space-fractional nonlocal model with variable diffusivity coefficients}, SIAM
  Journal on Numerical Analysis, 54 (2016), pp.~606--624.

\bibitem{chen2014generalized}
{\sc S.~Chen, J.~Shen, and L.~Wang}, {\em Generalized jacobi functions and
  their applications to fractional differential equations}, arXiv preprint
  arXiv:1407.8303,  (2014).

\bibitem{cho2015fractional}
{\sc Y.~Cho, I.~Kim, and D.~Sheen}, {\em A fractional-order model for minmod
  millennium}, Mathematical biosciences, 262 (2015), pp.~36--45.

\bibitem{Castillo2004Plasma}
{\sc D.~del Castillo-Negrete, B.~A. Carreras, and V.~E. Lynch}, {\em Fractional
  diffusion in plasma turbulence}, Physics of Plasmas (1994-present), 11
  (2004), pp.~3854--3864.

\bibitem{djordjevic2003fractional}
{\sc V.~D. Djordjevi{\'c}, J.~Jari{\'c}, B.~Fabry, J.~J. Fredberg, and
  D.~Stamenovi{\'c}}, {\em Fractional derivatives embody essential features of
  cell rheological behavior}, Annals of biomedical engineering, 31 (2003),
  pp.~692--699.

\bibitem{duarte2002chaotic}
{\sc F.~B. Duarte and J.~T. Machado}, {\em Chaotic phenomena and
  fractional-order dynamics in the trajectory control of redundant
  manipulators}, Nonlinear Dynamics, 29 (2002), pp.~315--342.

\bibitem{ervin2007variational}
{\sc V.~J. Ervin and J.~P. Roop}, {\em Variational solution of fractional
  advection dispersion equations on bounded domains in \uppercase{R}$^d$},
  Numerical Methods for Partial Differential Equations, 23 (2007), p.~256.

\bibitem{ghazizadeh2012inverse}
{\sc H.~R. Ghazizadeh, A.~Azimi, and M.~Maerefat}, {\em An inverse problem to
  estimate relaxation parameter and order of fractionality in fractional
  single-phase-lag heat equation}, International Journal of Heat and Mass
  Transfer, 55 (2012), pp.~2095--2101.

\bibitem{Gorenflo2002}
{\sc Rudolf Gorenflo, Francesco Mainardi, Daniele Moretti, and Paolo Paradisi},
  {\em Time fractional diffusion: a discrete random walk approach}, Nonlinear
  Dynamics, 29 (2002), pp.~129--143.

\bibitem{jaishankar2013}
{\sc Aditya Jaishankar and Gareth~H McKinley}, {\em Power-law rheology in the
  bulk and at the interface: quasi-properties and fractional constitutive
  equations}, Proceedings of the Royal Society A: Mathematical, Physical and
  Engineering Science, 469 (2013), p.~20120284.

\bibitem{jaishankar2014}
\leavevmode\vrule height 2pt depth -1.6pt width 23pt, {\em A fractional k-bkz
  constitutive formulation for describing the nonlinear rheology of multiscale
  complex fluids}, Journal of Rheology (1978-present), 58 (2014),
  pp.~1751--1788.

\bibitem{jha2003evidence}
{\sc R.~Jha, P.~K. Kaw, D.~R. Kulkarni, J.~C. Parikh, and ADITYA Team}, {\em
  Evidence of l{\'e}vy stable process in tokamak edge turbulence}, Physics of
  Plasmas (1994-present), 10 (2003), pp.~699--704.

\bibitem{kelly2017fracfit}
{\sc J.~F. Kelly, D.~Bolster, M.~M. Meerschaert, J.~D. Drummond, and A.~I.
  Packman}, {\em Fracfit: A robust parameter estimation tool for fractional
  calculus models}, Water Resources Research, 53 (2017), pp.~2559--2567.

\bibitem{Khader2011}
{\sc MM~Khader}, {\em On the numerical solutions for the fractional diffusion
  equation}, Communications in Nonlinear Science and Numerical Simulation, 16
  (2011), pp.~2535--2542.

\bibitem{Khader2012}
{\sc MM~Khader and AS~Hendy}, {\em The approximate and exact solutions of the
  fractional-order delay differential equations using legendre pseudospectral
  method}, International Journal of Pure and Applied Mathematics, 74 (2012),
  pp.~287--297.

\bibitem{kharazmi2018fractional}
{\sc E.~Kharazmi and M.~Zayernouri}, {\em Fractional pseudo-spectral methods
  for distributed-order fractional pdes}, International Journal of Computer
  Mathematics,  (2018), pp.~1--22.

\bibitem{kharazmi2017sem}
{\sc E.~Kharazmi, M.~Zayernouri, and G.~E. Karniadakis}, {\em A
  petrov–galerkin spectral element method for fractional elliptic problems},
  Computer Methods in Applied Mechanics and Engineering, 324.

\bibitem{kharazmi2017petrov}
\leavevmode\vrule height 2pt depth -1.6pt width 23pt, {\em Petrov--galerkin and
  spectral collocation methods for distributed order differential equations},
  SIAM Journal on Scientific Computing, 39 (2017), pp.~A1003--A1037.

\bibitem{le1991fractal}
{\sc A.~Le~M{\'e}haut{\'e}}, {\em Fractal Geometries Theory and Applications},
  CRC Press, 1991.

\bibitem{Li2009}
{\sc X.~Li and C.~Xu}, {\em A space-time spectral method for the time
  fractional diffusion equation}, SIAM Journal on Numerical Analysis, 47
  (2009), pp.~2108--2131.

\bibitem{Li2010}
\leavevmode\vrule height 2pt depth -1.6pt width 23pt, {\em Existence and
  uniqueness of the weak solution of the space-time fractional diffusion
  equation and a spectral method approximation}, Communications in
  Computational Physics, 8 (2010), p.~1016.

\bibitem{lim2014parameter}
{\sc C.~Y. Lim, M.~M. Meerschaert, and H.~P. Scheffler}, {\em Parameter
  estimation for operator scaling random fields}, Journal of Multivariate
  Analysis, 123 (2014), pp.~172--183.

\bibitem{Lin2007}
{\sc Y.~Lin and C.~Xu}, {\em Finite difference/spectral approximations for the
  time-fractional diffusion equation}, Journal of Computational Physics, 225
  (2007), pp.~1533--1552.

\bibitem{lischke2017petrov}
{\sc A.~Lischke, M.~Zayernouri, and G.~E. Karniadakis}, {\em A
  \uppercase{P}etrov--\uppercase{G}alerkin spectral method of linear complexity
  for fractional multiterm \uppercase{ODE}s on the half line}, SIAM Journal on
  Scientific Computing, 39 (2017), pp.~A922--A946.

\bibitem{liu2016two}
{\sc S.~Liu and R.~A. Canfield}, {\em Two forms of continuum shape sensitivity
  method for fluid--structure interaction problems}, Journal of Fluids and
  Structures, 62 (2016), pp.~46--64.

\bibitem{magin2010fractional}
{\sc R.~L. Magin}, {\em Fractional calculus models of complex dynamics in
  biological tissues}, Computers \& Mathematics with Applications, 59 (2010),
  pp.~1586--1593.

\bibitem{Mainardi2010}
{\sc Francesco Mainardi}, {\em Fractional calculus and waves in linear
  viscoelasticity: an introduction to mathematical models}, Imperial College
  Press, 2010.

\bibitem{Martins00}
{\sc J.~Martins, I.~Kroo, and J.~Alonso}, {\em An automated method for
  sensitivity analysis using complex variables}, in Proceedings of the 38th
  Aerospace Sciences Meeting, Reno, NV, 2000, pp.~{AIAA} 2000--0689.

\bibitem{meral2010fractional}
{\sc F.~C. Meral, T.~J. Royston, and R.~Magin}, {\em Fractional calculus in
  viscoelasticity: an experimental study}, Communications in Nonlinear Science
  and Numerical Simulation, 15 (2010), pp.~939--945.

\bibitem{Miller93}
{\sc K.~S. Miller and B.~Ross}, {\em An Introduction to the Fractional Calculus
  and Fractional Differential Equations}, New York, NY:John Wiley and Sons,
  Inc., 1993.

\bibitem{naghibolhosseini2015estimation}
{\sc M.~Naghibolhosseini}, {\em Estimation of outer-middle ear transmission
  using \uppercase{DPOAE}s and fractional-order modeling of human middle ear},
  PhD thesis, City University of New York, NY., 2015.

\bibitem{naghibolhosseini2018fractional}
{\sc M.~Naghibolhosseini and G.~R. Long}, {\em Fractional-order modelling and
  simulation of human ear}, International Journal of Computer Mathematics, 95
  (2018), pp.~1257--1273.

\bibitem{Podlubny99}
{\sc I~Podlubny}, {\em Fractional Differential Equations}, San Diego, CA, USA:
  Academic Press, 1999.

\bibitem{Rawashdeh2006}
{\sc EA~Rawashdeh}, {\em Numerical solution of fractional integro-differential
  equations by collocation method}, Applied mathematics and computation, 176
  (2006), pp.~1--6.

\bibitem{samiee2017fast}
{\sc M.~Samiee, E.~Kharazmi, and M.~Zayernouri}, {\em Fast spectral methods for
  temporally-distributed fractional \uppercase{PDE}s}, in Spectral and High
  Order Methods for Partial Differential Equations ICOSAHOM 2016, Springer,
  2017, pp.~651--667.

\bibitem{samiee2018petrov}
{\sc M.~Samiee, E.~Kharazmi, M.~Zayernouri, and M.~M. Meerschaert}, {\em
  \uppercase{P}etrov-\uppercase{G}alerkin method for fully distributed-order
  fractional partial differential equations}, arXiv preprint arXiv:1805.08242,
  (2018).

\bibitem{Samiee2016unified}
{\sc M.~Samiee, Zayernouri M., and M.~M. Meerschaert}, {\em A unified spectral
  method for fractional \uppercase{PDE}s with two-sided derivatives; part
  \uppercase{i}: A fast solver}, Journal of Computational Physics (in press),
  (2017).

\bibitem{samiee2017unified}
{\sc M.~Samiee, M.~Zayernouri, and M.~M. Meerschaert}, {\em A unified spectral
  method for fractioanl \uppercase{PDE}s with two-sided derivatives; stability,
  and error analysis}, arXiv preprint arXiv:1710.08337,  (2017).

\bibitem{Sobieski90}
{\sc J.~S. Sobieski}, {\em Sensitivity of complex, internally coupled systems},
  AIAA Journal,  (1990).

\bibitem{sreenivasan1997phenomenology}
{\sc K.~R. Sreenivasan and R.~A. Antonia}, {\em The phenomenology of
  small-scale turbulence}, Annual review of fluid mechanics, 29 (1997),
  pp.~435--472.

\bibitem{stanford2010adjoint}
{\sc B.~Stanford, P.~Beran, and M.~Kurdi}, {\em Adjoint sensitivities of
  time-periodic nonlinear structural dynamics via model reduction}, Computers
  \& structures, 88 (2010), pp.~1110--1123.

\bibitem{Sun2006}
{\sc Z.~Sun and X.~Wu}, {\em A fully discrete difference scheme for a
  diffusion-wave system}, Applied Numerical Mathematics, 56 (2006),
  pp.~193--209.

\bibitem{suzuki2016fractional}
{\sc J.~L. Suzuki, M.~Zayernouri, M.~L. Bittencourt, and G.~E. Karniadakis},
  {\em Fractional-order uniaxial visco-elasto-plastic models for structural
  analysis}, Computer Methods in Applied Mechanics and Engineering, 308 (2016),
  pp.~443--467.

\bibitem{van2005review}
{\sc F.~Van~Keulen, R.~T. Haftka, and N.~H. Kim}, {\em Review of options for
  structural design sensitivity analysis. part 1: Linear systems}, Computer
  methods in applied mechanics and engineering, 194 (2005), pp.~3213--3243.

\bibitem{wang2010direct}
{\sc H.~Wang, K.~Wang, and T.~Sircar}, {\em A direct $o (n log^2 n)$ finite
  difference method for fractional diffusion equations}, Journal of
  Computational Physics, 229 (2010), pp.~8095--8104.

\bibitem{wang2015high}
{\sc H.~Wang and X.~Zhang}, {\em A high-accuracy preserving spectral galerkin
  method for the dirichlet boundary-value problem of variable-coefficient
  conservative fractional diffusion equations}, Journal of Computational
  Physics, 281 (2015), pp.~67--81.

\bibitem{wang2011fast}
{\sc K.~Wang and H.~Wang}, {\em A fast characteristic finite difference method
  for fractional advection--diffusion equations}, Advances in \uppercase{w}ater
  \uppercase{r}esources, 34 (2011), pp.~810--816.

\bibitem{wei2010coupled}
{\sc H.~Wei, W.~Chen, H.~Sun, and X.~Li}, {\em A coupled method for inverse
  source problem of spatial fractional anomalous diffusion equations}, Inverse
  Problems in Science and Engineering; Formerly Inverse Problems in
  Engineering, 18 (2010), pp.~945--956.

\bibitem{west2016fractional}
{\sc B.~J. West}, {\em Fractional Calculus View of Complexity: Tomorrow’s
  Science}, CRC Press, 2016.

\bibitem{West2003}
{\sc B.~J. West, M.~Bologna, and P.~Grigolini}, {\em Physics of Fractal
  Operators}, New York, NY: Springer Verlag., 2003.

\bibitem{yu2016numerical}
{\sc B.~Yu and X.~Jiang}, {\em Numerical identification of the fractional
  derivatives in the two-dimensional fractional cable equation}, Journal of
  Scientific Computing, 68 (2016), pp.~252--272.

\bibitem{yu2017numerical}
{\sc B.~Yu, X.~Jiang, and H.~Qi}, {\em Numerical method for the estimation of
  the fractional parameters in the fractional mobile/immobile
  advection--diffusion model}, International Journal of Computer Mathematics,
  (2017), pp.~1--20.

\bibitem{zamani2015asymmetric}
{\sc V.~Zamani, E.~Kharazmi, and R.~Mukherjee}, {\em Asymmetric post-flutter
  oscillations of a cantilever due to a dynamic follower force}, Journal of
  Sound and Vibration, 340 (2015), pp.~253--266.

\bibitem{zayernouri2015tempered}
{\sc M.~Zayernouri, M.~Ainsworth, and G.~E. Karniadakis}, {\em Tempered
  fractional sturm--liouville eigenproblems}, SIAM Journal on Scientific
  Computing, 37 (2015), pp.~A1777--A1800.

\bibitem{Zayernouri14-SIAM-Frac-Delay}
{\sc M.~Zayernouri, W.~Cao, Z.~Zhang, and G.~E. Karniadakis}, {\em Spectral and
  discontinuous spectral element methods for fractional delay equations}, SIAM
  Journal on Scientific Computing, 36 (2014), pp.~B904--B929.

\bibitem{Zayernouri2013}
{\sc M.~Zayernouri and G.~E. Karniadakis}, {\em Fractional
  \uppercase{S}turm-\uppercase{L}iouville eigen-problems: theory and numerical
  approximations}, J. Comp. Physics, 47-3 (2013), pp.~2108--2131.

\bibitem{Zayernouri14-SIAM-Frac-Advection}
\leavevmode\vrule height 2pt depth -1.6pt width 23pt, {\em Discontinuous
  spectral element methods for time- and space-fractional advection equations},
  SIAM Journal on Scientific Computing, 36 (2014), pp.~B684--B707.

\bibitem{Zayernouri_FODEs_2014}
\leavevmode\vrule height 2pt depth -1.6pt width 23pt, {\em Exponentially
  accurate spectral and spectral element methods for fractional odes}, J. Comp.
  Physics, 257 (2014), pp.~460--480.

\bibitem{zayernouri2016_JCP_Frac_AB_AM}
{\sc M.~Zayernouri and A.~Matzavinos}, {\em Fractional
  \uppercase{A}dams-\uppercase{B}ashforth/\uppercase{M}oulton methods: An
  application to the fractional \uppercase{K}eller--\uppercase{S}egel
  chemotaxis system}, Journal of Computational Physics-In Press,  (2016).

\bibitem{zayernouri2011coherent}
{\sc M~Zayernouri and M~Metzger}, {\em Coherent features in the sensitivity
  field of a planar mixing layer}, Physics of Fluids (1994-present), 23 (2011),
  p.~025105.

\bibitem{zeng2015numerical}
{\sc F.~Zeng, C.~Li, F.~Liu, and I.~Turner}, {\em Numerical algorithms for
  time-fractional subdiffusion equation with second-order accuracy}, SIAM
  Journal on Scientific Computing, 37 (2015), pp.~A55--A78.

\end{thebibliography}

\end{document}